\documentclass[11pt, a4paper]{amsart}

\usepackage[latin1]{inputenc}
\usepackage[T1]{fontenc}
\usepackage[english]{babel}
\usepackage{amssymb, amsmath, amsthm}
\usepackage[colorlinks=true, allcolors=blue]{hyperref}
\usepackage{geometry}
\usepackage{fancyhdr}
\usepackage{relsize}
\usepackage{upgreek}
\usepackage{verbatim}
\usepackage{enumitem}
\usepackage{todonotes}
\usepackage{tikz}
\usetikzlibrary{matrix,arrows,calc}
\usepackage{graphicx}
\usepackage{package}
\usepackage[all]{xy}

\newcommand{\CC}{\mathbb{C}}

\newcommand{\EE}{\mathcal{E}}
\newcommand{\FF}{\mathcal{F}}
\newcommand{\QQ}{\mathcal{Q}}

\newcommand{\ZZ}{\mathbb{Z}}
\newcommand{\PP}{\mathbf{P}}

\newcommand{\rad}{\mathrm{rad}}

\newtheorem{definition}[thm]{Definition}

\title{Cell decompositions and algebraicity of cohomology for quiver Grassmannians}
\author{Giovanni Cerulli Irelli \and Francesco Esposito \and Hans Franzen \and Markus Reineke}
\address{
Giovanni Cerulli Irelli\\ Sapienza-Universit\`a di Roma\\ Dipartimento S.B.A.I.\\Via A. Scarpa 10 \\00161 Roma (Italy)} \email{giovanni.cerulliirelli@uniroma1.it} 
\address{ 
Francesco Esposito\\ Universit\`a di Padova\\ Dipartimento di Matematica\\ Via Trieste 63\\ 35121 Padova (Italy)}
\email{esposito@math.unipd.it}
\address{
Hans Franzen\\ Ruhr-University Bochum\\ Faculty of Mathematics\\
Universit\"atsstrasse 150\\
44780 Bochum (Germany)}
\email{hans.franzen@rub.de}
\address{ 
Markus Reineke\\ Ruhr-University Bochum\\ Faculty of Mathematics\\
Universit\"atsstrasse 150\\
44780 Bochum (Germany)}
\email{Markus.Reineke@ruhr-uni-bochum.de}

\begin{document}
	
	\begin{abstract}
We show that the cohomology ring of a quiver Grassmannian asssociated with a rigid quiver representation has property (S): there is no odd cohomology and  the cycle map is an isomorphism; moreover, its Chow ring admits  explicit generators defined over any field. From this we deduce the polynomial point count property. By restricting the quiver to finite or affine type, we are able to show a much stronger assertion: namely, that a quiver Grassmannian associated to an indecomposable (not necessarily rigid) representation  admits a cellular decomposition. As a corollary, we establish a cellular decomposition for quiver Grassmannians associated with representations with rigid regular part. Finally, we study the geometry behind the cluster multiplication formula of Caldero and Keller, providing a new proof of a slightly more general result. 
\end{abstract}
\maketitle
	\tableofcontents
	
	\section{Introduction}
	It has long been known that the representation theory of certain algebraic objects (such as Weyl groups, semisimple Lie algebras, Kac--Moody algebras) is intimately linked to the topology of suitable varieties  (such as Grassmannians, flag varieties, Schubert varieties, Springer fibers). Relatively more recently, varieties associated to representations of quivers have been used by Lusztig for the study of canonical bases of quantum groups, building on work of Ringel who realized the upper part of  quantum groups as Hall algebras of quiver representations over finite fields. In 2001, Fomin and Zelevinsky \cite{FZ1} introduced cluster algebras for the study of dual canonical bases of quantum groups. It was hence expected that there exists a close connection between Hall algebras and cluster algebras. This was realized first by Caldero and Chapoton \cite{CC} in 2006 and extended to full generality by Caldero--Keller \cite{CK1,CK2}, Palu \cite{Palu}, Derksen-Weyman-Zelevinsky \cite{DWZ}: instead of using Hall numbers, Caldero and Chapoton used the Euler characteristic of complex quiver Grassmannians to realize the generators of the cluster algebras. The key of Caldero and Chapoton's realization was a multiplication formula for certain characters of the Grothendieck group of $\textrm{Rep}_K(Q)$ where $Q$ is a Dynkin quiver. This formula is based on the non--trivial fact that quiver Grassmannians for Dynkin quivers have polynomial point count, as a consequence of Ringel's work on Hall numbers. The multiplication formula was then extended by Caldero--Keller \cite{CK1,CK2} and Palu \cite{Palu, Palu:12}. A different approach for the proof of a multiplication formula was achieved by Hubery \cite{Hubery} (and independently by Xu \cite{Xu}) and it is based on Green's theorem and associativity of Hall numbers. 
	
	The positivity conjecture of Fomin--Zelevinsky translates into the positivity conjecture for the Euler--Poincar\'e characterisitic of the quiver Grassmannians associated with rigid representations of acyclic quivers (i.e. quiver representations without self-extensions). This result was proved  by Nakajima in \cite[Th.~A.1]{Nakajima:Cluster} and, independently, by Qin in \cite{Qin}: They show that such quiver Grassmannians do not have odd cohomology. 
	
	Following \cite{DLP}, we say that a smooth projective complex variety $X$ has property $(S)$ if
	\begin{enumerate} 
		\item numerical and rational equivalence on $X$ coincide, in particular the Chow ring $A^*(X)$ is a finitely generated  free abelian group,
		\item $H^{2i+1}(X)=0$ for every $i$, and
		\item the cycle map $A^i(X) \to H^{2i}(X)$ is an isomorphism for all $i$
	\end{enumerate}  
	(see Section~\ref{Sec:PropertySandCellDec}). For a (finite dimensional) quiver representation $M$ we denote by $\Gr_\mathbf{e}(M)$ the projective variety parametrizing subrepresentations of $M$ of dimension vector $\mathbf{e}$, i.e. a quiver Grassmannian attached to $M$. We say that $M$ has property (S) if $\Gr_\mathbf{e}(M)$ has property (S) for all $\mathbf{e}$.  
	\begin{thm} \label{t:SIntro}
A rigid quiver representation $M$ has property (S). 
\end{thm}
It is known that the support of a rigid quiver representation is acyclic (see e.g. \cite[Prop.~2]{Ringel:Tree}) and hence, to prove theorem~\ref{t:SIntro},  it is not restrictive to assume to work with acyclic quivers.  
The proof of Theorem~\ref{t:SIntro} is given in Section~\ref{Sec:PropertyS} and is based on the Ellingsrud-Str{\o}mme decomposition of the diagonal, which provides explicit generators of the Chow rings. Moreover, we prove that quiver Grassmannians attached to rigid quiver representations are irreducible over any field and smooth over $\mathbb{Z}$ (see Proposition~\ref{p:irred_rigid} and Lemma~\ref{Lem:SchemeZ}). This has the following consequence that was previously proved by Qin in \cite[Th.~3.2.6]{Qin} employing quantum cluster algebras (see Section~\ref{Subsec:Cor2}).
\begin{cor}\label{Cor:PolyCount}
A quiver Grassmannian attached to a rigid representation has polynomial point--count.
\end{cor}
Quiver Grassmannians play a central role in cluster theory. In particular, it is important to know the vanishing of odd cohomology (for positivity) and the polynomial count (for quantum cluster variables). These properties were obtained by Nakajima \cite{Nakajima:Cluster} and Qin, \cite{Qin} with different techniques. We obtain these results by noticing that these varieties admit the decomposition of the diagonal \cite{Voisin}. 

We recall (see Section~\ref{Sec:PropertySandCellDec}) that a variety $X$ admits a cellular decomposition, or affine paving, if there exists a finite partition $X=X_1\amalg\cdots\amalg X_h$ such that each piece $X_i$ is an affine space, and $X_1\cup\cdots\cup X_i$ is closed in $X$ for every $i=1,\cdots, h$. 
Cellular decomposition implies property (S), but in general they are not equivalent. We say that a (finite-dimensional) quiver representation $M$ has property (C) if $\Gr_\mathbf{e}(M)$ admits a cellular decomposition for all $\mathbf{e}$.   Property (C)  has been shown to hold for some rigid representations of an oriented tree \cite{Lorscheid} and for rigid representations of a generalized Kronecker quiver \cite{RW}. 
\begin{conj}\label{Conj:Rigid}
Any rigid quiver representation has property (C). 
\end{conj}
For quivers of finite or affine type, the following much stronger statement is true and new. 

\begin{thm} \label{t:DynkinIntro}
An indecomposabe representation of a quiver of finite or affine type has property (C).  
\end{thm}
Theorem~\ref{t:DynkinIntro} was proved in \cite[Th.~12]{CFFFR} for a quiver of type  $A$-equioriented, in  \cite{CEsp} for the Kronecker quiver, in \cite{LW1, LW2} for preprojective representations of  type $\tilde{D}_n$, it follows from \cite{CI2011} for rigid representations of type $\tilde{A}_n$ as a consequence of the existence of a torus action.

The proof of Theorem~\ref{t:DynkinIntro} is based on Theorem~\ref{Thm:Reduction2}, which we now illustrate (see Section~\ref{Sec:ReductionThms} for details). Let $Q$ be an acyclic quiver (i.e. a quiver without oriented cycles), and let $X$ and $S$ be two $Q$--representations. Any short exact sequence $\xi: 0\rightarrow X\rightarrow Y\rightarrow S\rightarrow 0$ gives rise to a map $\Psi^\xi: \Gr_\mathbf{e}(Y)\rightarrow\coprod_{\mathbf{f+g=e}}\Gr_\mathbf{f}(X)\times\Gr_\mathbf{g}(S)$ which restricts to algebraic maps: 
$$
\Psi_\mathbf{f,g}^\xi:\mathcal{S}_\mathbf{f,g}^\xi\rightarrow\Gr_\mathbf{f}(X)\times\Gr_\mathbf{g}(S)\;\; (\textrm{for every }\mathbf{f+g=e})
$$
where $\mathcal{S}_\mathbf{f,g}^\xi$ is the preimage of $\Gr_\mathbf{f}(X)\times\Gr_\mathbf{g}(S)$ under $\Psi^\xi$. 
We prove that if $\xi$ generates $\Ext^1(S,X)$, then this is a Zariski-locally trivial affine bundle over its image. So, property (C) of $Y$ can be deduced by the analysis of the images of such maps. 
If $\Ext^1(S,X)=0$, property (C) (resp.  property (S)) for $X$ and $S$ implies property (C) (resp.  property (S)) for $Y$. So Theorems~\ref{t:SIntro} and \ref{t:DynkinIntro} imply the following:
\begin{cor}\label{Cor:DynkAffIntro}Property~(S) holds for a $Q$--representation $M$ with rigid regular part.
Moreover,  property~(C) holds for $M$
if $Q$ is of finite or affine type. 
\end{cor}
So every representation of a Dynkin quiver has property (C) (see Section~\ref{Sec:CellDecDynkin}). We conjecture the same for  affine quivers, i.e. that property (C) holds for regular decomposable representations of affine quivers. For a wild quiver, property (C) cannot hold for every representation. Indeed Ringel has recently shown \cite{Ringel:QuivGrass} the following striking result: given any wild quiver $Q$ and any projective variety $X$ there exists a representation $N$ of $Q$ and a dimension vector $\mathbf{e}$ such that $X\cong Gr_\mathbf{e}(N)$.

If $\xi$ generates $\Ext^1(S,X)\neq0$, the maps $\Psi_\mathbf{f,g}^\xi$ are not surjective and we determine their images. We prove that the following are well--defined (see Lemma~\ref{Lem:XsSXProperties}): 
$$
\begin{array}{cc}
X_S:=\textrm{max}\{N\subseteq_Q X|\, [S, X/N]^1=1\},&S^X:=\textrm{min}\,\{N\subseteq_QS|\, [N,X]^1=1\},
\end{array}
$$ 
where we use the standard shorthand notation $[A,B]^1:=\textrm{dim}\Ext^1(A,B)$.
It turns out that the image of $\Psi^\xi_\mathbf{f,g}$ is the complement inside  $\Gr_\mathbf{f}(X)\times\Gr_\mathbf{g}(S)$ of the closed subset $\left(\Gr_\mathbf{f}(X_S)\times\Gr_{\mathbf{g-dim}\,S^X}(S/S^X)\right)$. 
From this description, we deduce a multiplication formula for the cluster character with coefficients $CC:\textrm{Rep}_K(Q)\rightarrow \ZZ[y_1,\cdots, y_n, x_1^{\pm1},\cdots, x_n^{\pm 1}]$  (see Section~\ref{Sec:ClusterMultFormula} for details).

\begin{thm}\label{Thm:ClusterIntro} Under the previous hypothesis
$
CC(X)CC(S)=CC(Y)+\mathbf{y}^{\mathbf{dim}\, S^X} CC(X_S\oplus S/S^X)\mathbf{x}^{\mathbf{f}}.
$
\end{thm}
Multiplication formulas for cluster characters were studied intensively by several authors \cite{CK1, CK2, Palu, Palu:12, Hubery, Xu} in much more general contexts. In \cite{CK1, Palu:12, Hubery, Xu} the assumption $[S,X]^1=1$ is dropped,
 and it would be interesting to know if there exists a more general reduction theorem, underlying those results.

\section{Generalities}\label{Sec:Generalities}
	
\subsection{Representation theory of quivers}\label{Sec:Quivers}
In this section we collect some well-known facts about quiver representations. Standard references are \cite{CB1}, \cite{CB2}, \cite{ARS:97}, \cite{Ringel:Book}, \cite{ASS:06}. Let $Q=(Q_0,Q_1,s,t)$ be a finite connected quiver with set of vertices $Q_0$ of cardinality $n$, finite set of arrows $Q_1$ and every arrow $\alpha\in Q_1$ is oriented  from its starting vertex $s(\alpha)\in Q_0$ towards its terminal vertex $t(\alpha)\in Q_0$ and we write $\alpha:s(\alpha)\rightarrow t(\alpha)$. Given a field $K$, we consider the category $\textrm{Rep}_K(Q)$ of $Q$--representations over $K$. For a representation $M$ of $Q$ we denote by $M_i$ the $K$--vector space attached to vertex $i\in Q_0$, and by $M_\alpha: M_{s(\alpha)}\rightarrow M_{t(\alpha)}$ the K--linear map attached to an arrow $\alpha\in Q_1$.  The category $\textrm{Rep}_K(Q)$ of finite--dimensional $Q$--representations is equivalent to the category $A\!-\!\textrm{mod}$ of finite dimensional module over the path algebra $A=KQ$ of $Q$.  We work with left modules and we do not distinguish between objects of $\textrm{Rep}_K(Q)$ and $A\!-\!\textrm{mod}$. The category $\textrm{Rep}_K(Q)$ is abelian, Krull-Schmidt and hereditary, i.e. every modules has projective dimension at most one, i.e. the functors $\Ext^i_Q(-,-)$ vanish for $i\geq2$. Given two $Q$--representations, we often use the standard notation:  
$$
\begin{array}{cc}
[M,N]:=\textrm{dim}_K\Hom_Q(M,N),&[M,N]^1:=\textrm{dim}_K\Ext^1_Q(M,N).
\end{array}
$$
When $Q$ is acyclic, the simple objects of $\textrm{Rep}_K(Q)$ are one--dimensional and supported on a single vertex. Given a every vertex $k\in Q_0$ we denote by $S_k$ the corresponding simple, by $P_k$ its projective cover  and by $I_k$ its injective envelope. The Nakayama functor  $\nu:=D\Hom(-,A)$ establishes an equivalence $\nu:\textrm{Proj}(A)\rightarrow \textrm{Inj}(A)$ from the category of projectives to the category of injectives, and it is characterized by $\nu(P_k)=I_k$. \subsubsection{Auslander-Reiten theory of $\textrm{Rep}_K(Q)$} The category $\textrm{Rep}_K(Q)$ is endowed with the two endofunctors  $\tau:=D\Ext^1(-, A)$ and $\tau^-:=\Ext^1(D(-),A)$, the Auslander-Reiten translate and its quasi--inverse ($D=\Hom_K(-,K)$ denotes the standard duality). Notice that $\tau$ is left-exact and $\tau^-$ is right-exact. For any two $A$--modules $L$ and $M$ there are functorial isomorphisms called Auslander-Reiten formulas:
$$
\Ext^1(L,M)\cong D\Hom(M,\tau L)\cong D\Hom(\tau^- M, L).
$$
Given an indecomposable non--projective $Q$--representation $M$ there exists a short exact sequence 
$$
0\rightarrow \tau M\rightarrow E\rightarrow M\rightarrow 0
$$
which is almost split, i.e. it is non--split, every morphism $\tau M\rightarrow Z$ which is not split epi factors through $E$ and every morphism $Z\rightarrow M$ which is not split mono factors through $E$. 

An indecomposable module $M$ is called preprojective (resp. preinjective) if there exist $k\geq0$ and $\ell\in Q_0$ such that $M\cong \tau^{-k}P_\ell$ (resp. $M\cong \tau^k I_\ell$) and  regular if $\tau^{\pm k} M$ is non--zero for all $k$. We say that $M$ is  preprojective (resp. preinjective, regular) if all its indecomposable direct summands are  preprojective (resp. preinjective, regular). We denote respectively by $\mathcal{P}$, $\mathcal{R}$, $\mathcal{I}$ the full subcategory of $A$--mod whose objects are preprojectives, regular, preinjectives, respectively. The category $\mathcal{P}$ is closed under taking submodules; the category $\mathcal{I}$ is closed under taking quotients. Moreover, 
$$
\Hom(\mathcal{I}, \mathcal{R})=\Hom(\mathcal{I}, \mathcal{P})=\Hom(\mathcal{R}, \mathcal{P})=0=
\Ext^1(\mathcal{R}, \mathcal{I})=\Ext^1(\mathcal{P}, \mathcal{I})=\Ext^1(\mathcal{P}, \mathcal{R}).
$$
Every module $M$ admits a unique split filtration
$M'\subseteq M''\subseteq M$
where $M'\in\mathcal{I}$, $M''/M'\in\mathcal{R}$ and $M/M''\in\mathcal{P}$; these are called the preinjective, regular and preprojective parts of $M$, respectively.  

Given two indecomposable $Q$--representations $X$ and $Y$, let $\rad(X,Y)\subseteq\Hom_Q(X,Y)$ be the vector subspace of non--invertible morphisms from $X$ to $Y$. Inside $\rad(X,Y)$ there is the  subspace $\rad^2(X,Y)$ of morphisms $f=f_2\circ f_1$ such that $f_1\in \rad(X,M)$ and $f_2\in \rad(M,Y)$ for some $M$. The quotient space is denoted by $\mathrm{Irr}(X,Y)=\rad(X,Y)/\rad^2(X,Y)$. A morphism $f\in\Hom_Q(X,Y)$ is called irreducible if $f\in \rad(X,Y)\setminus \rad^2(X,Y)$. The Auslander-Reiten quiver $\Gamma(Q)$ of $Q$ is the quiver whose vertices are isoclasses of indecomposable $Q$--representations and the number of arrows between two vertices $[X]$ and $[Y]$ (corresponding to the two indecomposables $X$ and $Y$) equals $\textrm{dim}_K\textrm{Irr}(X,Y)$.  Many properties of $\textrm{Rep}_K(Q)$ can be read from $\Gamma(Q)$. The quiver $\Gamma(Q)$ is finite if and only if $Q$ is Dynkin. 
If $Q$ is not Dynkin, $\Gamma(Q)=\Gamma(\mathcal{P})\cup \Gamma(\mathcal{R})\cup\Gamma(\mathcal{I})$ where $\Gamma(\mathcal{P})$ (resp. $\Gamma(\mathcal{R})$, resp. $\Gamma(\mathcal{I})$) is the Auslander--Reiten quiver of $\mathcal{P}$ (resp. $\mathcal{R}$, resp. $\mathcal{I}$). Notice that $\Gamma(\mathcal{P})$ (resp. $\Gamma(\mathcal{I})$) contains all the indecomposable projectives (resp. injectives), and can be described combinatorially via the knitting algorithm. The regular components are described by Ringel \cite{Ringel:Wild}. 

Given two indecomposable representations $X$ and $Y$, a sectional morphism $f:X\rightarrow Y$ is a composition 
$$
f: \xymatrix{X\ar^{f_1}[r]&X_1\ar^{f_2}[r]&\cdots\ar^{f_{t-1}}[r]&X_{t-1}\ar^{f_t}[r]&X_t=Y}
$$
of irreducible morphisms $f_i:X_{i-1}\rightarrow X_i$ such that  $X_{i}\not\cong 
\tau X_{i+2}$. Every sectional morphism is represented by an oriented (connected) path between $[X]$ and $[Y]$ in $\Gamma(Q)$ which does not contain a path of the form $[\tau M]\rightarrow [Z]\rightarrow [M]$. 

\subsubsection{Euler form and representation type}
Given a representation $M$ of $Q$ we denote by $\mathbf{dim}\,M:=(\textrm{dim}_K(M_i))\in\ZZ_{\geq0}^{Q_0}$ the dimension vector of $M$ and we sometimes use the shorthand notation $\mathbf{d}^M=(d_i^M)$.  Given two $Q$--representations $M$ and $N$, Ringel defined the following map 
\begin{equation}\label{Eq:DefiPhiQuiver}
\Phi_M^N:\bigoplus_{i\in Q_0} \Hom_K(M_i,N_i) \rightarrow \bigoplus_{\alpha\in Q_1}\Hom_K(M_{s(\alpha)}, N_{t(\alpha)}):\; (f_i)\mapsto (N_\alpha\circ f_{s(\alpha)}-f_{t(\alpha)}\circ M_\alpha).
\end{equation}
On sees that $\textrm{Ker}(\Phi_M^N)=\Hom_Q(M,N)$ and  $\textrm{Coker}(\Phi_M^N)\cong \Ext^1_Q(M,N)$. In particular we get
\begin{equation}\label{Eq:EulerForm}
[M,N]-[M,N]^1=\langle\mathbf{dim}\,M,\mathbf{dim}\,N\rangle:=\sum_{i\in Q_0}d_i^Md_i^N-\sum_{\alpha\in Q_1}d_{s(\alpha)}^Md_{t(\alpha)}^N.
\end{equation}
The bilinear form $\langle-,-\rangle:\ZZ^{Q_0}\times \ZZ^{Q_0}\rightarrow \ZZ$ is called the Euler--Ringel form of the quiver $Q$.  The corresponding symmetric form $b_Q(\mathbf{x},\mathbf{y}):=\frac{1}{2}(\langle\mathbf{x},\mathbf{y}\rangle+\langle\mathbf{y},\mathbf{x}\rangle)$ defines a quadratic form $q_Q(\mathbf{x})=b_Q(\mathbf{x},\mathbf{x})$ which plays an important r\^ole for the representation theory of $Q$. Indeed $q_Q$ is positive definite if and only if $Q$ is an orientation of  a simply--laced Dynkin diagram of type $A_n$ ($n\geq 1$), $D_n$ ($n\geq4$), $E_6$, $E_7$ and $E_8$. The quiver $Q$ is called affine if $q_Q$ is positive semi--definite, but not definite. This happens if and only if $Q$ is an orientation of  a simply--laced extended Dynkin diagram of type $\tilde{A}_n$ ($n\geq 1$), $\tilde{D}_n$ ($n\geq4$), $\tilde{E}_6$, $\tilde{E}_7$ and $\tilde{E}_8$ (see table~\ref{Fig:ExtendedDynkinDiagrams}). In this case the kernel of $q_Q$ is generated by the minimal positive imaginary that is denoted with $\delta$. Notice that $\delta$ does not depend on the orientation of the quiver, but only on its underlying graph.  In table~\ref{Fig:ExtendedDynkinDiagrams} we recollect the minimal positive imaginary root in each type. The quiver $Q$ is called wild if $q_Q$ is indefinite.  

By Gabriel's theorem, $Q$ admits a finite number of indecomposable representations (up to isomorphism) precisely when $q_Q$ is positive definite. In this case the dimension vectors of the indecomposables are precisely $\mathbf{x}\in\ZZ^{Q_0}_{\geq0}$ such that $q_Q(\mathbf{x})=1$. These dimension vectors are precisely the positive roots of the  semisimple complex Lie algebra corresponding to $Q$.

If $Q$ is affine, the dimension vectors of the indecomposable $Q$--representations are precisely those $\mathbf{x}\in\ZZ^{Q_0}_{\geq0}$ such that $q_Q(\mathbf{x})\leq 1$. Moreover $q_Q(\mathbf{x})=0$ if and only if $\mathbf{x}$ is a multiple of $\delta$. In section~\ref{Sec:Affine}, we will give more information concerning the representation theory of affine quivers. 

\subsubsection{Representation varieties and group actions}
A $Q$--representation $M$ is called rigid if $[M,M]^1=0$. It is called exceptional, if it is indecomposable and rigid. Finally, $M$ is called a brick if $[M,M]=1$. 
Given a dimension vector $\mathbf{d}\in\ZZ_{\geq0}^{Q_0}$, the vector space $R_\mathbf{d}(Q):=\bigoplus_{\alpha\in Q_1} \Hom_K(K^{d_{s(\alpha)}}, K^{d_{t(\alpha)}})$ parametrizes the $Q$--representations of dimension vector $\mathbf{d}$ and  we identify points in $R_\mathbf{d}$ with representations. This vector space is acted upon by the group $G_\mathbf{d}:=\prod_{i\in Q_0} GL_{d_i}(K)$ via base change. The $G_\mathbf{d}$--orbits are in bijection with isoclasses of $Q$--representations of dimension vector $\mathbf{d}$. Given a point $M\in R_\mathbf{d}(Q)$, the corresponding orbit has codimension $[M,M]^1$ in $\mathrm{R}_\mathbf{d}(Q)$. In particular, a rigid representation has dense orbit and hence for any dimension vector $\mathbf{d}$ there exists at most one rigid representation (up to isomorphism) that we denote by $M_\mathbf{d}$.

\subsection{Useful Lemmata}

We recall two  lemmata which are known to experts. We include  the proofs for convenience of the reader.

\begin{lem}[Happel--Ringel]\cite[Lem.~4.1]{HR:82}\label{Lem:HappelRingel}
Let $X$ and $Y$ be indecomposable $Q$--representations. If $[Y,X]^1=0$, then a non--zero map $f:X\rightarrow Y$ is either injective or surjective. 
\end{lem}
\begin{proof}
We consider the two short exact sequences
$$
\begin{array}{cc}
\xymatrix@C=15pt{
\xi_1:0\ar[r]&\ker f\ar[r]&X\ar[r]&\im f\ar[r]&0,
}
&
\xymatrix@C=15pt{
\xi_2:0\ar[r]&\im f\ar^\iota[r]&Y\ar[r]&\coker f\ar[r]&0.
}
\end{array}
$$
Apply $\Hom(-,\ker f)$ to $\xi_2$ and get a surjective morphism $\xymatrix{\Ext^1(Y,\ker f)\ar@{->>}[r]&\Ext^1(\im f,\ker f)}$  which implies the existence of a commutative diagram with exact rows
$$
\xymatrix@C=15pt{
\xi_1:0\ar[r]&\ker f\ar[r]\ar@{=}[d]&X\ar[r]\ar[d]&\im f\ar[r]\ar^\iota[d]&0\\
0\ar[r]&\ker f\ar[r]&Z\ar[r]&Y\ar[r]&0
}
$$
and hence of a short exact sequence
$$
\xymatrix{
0\ar[r]&X\ar[r]&Z\oplus \im f\ar[r]&Y\ar[r]&0.
}
$$
Since $\Ext^1(Y,X)=0$, by the Krull-Schmidt property, $\im f=Y$ or $\im f\cong X$. 
\end{proof}
The following lemma is attibuted to Unger, but we could not find a published version of it. 
\begin{lem}[Unger]\label{Lem:Unger}
Let $X$ and $Y$ be exceptional representations of $Q$ such that $[Y,X]^1=0$. Let $f:X\hookrightarrow Y$ be a monomorphism between them and let $S$ be its cokernel. Then $[S,S]=1$. In particular $S$ is indecomposable. Moreover 
\begin{equation}\label{Eq:Unger}
[S,S]^1=[X,Y]-1.
\end{equation}
\end{lem}
\begin{proof}
Since $Y$ is exceptional and $S$ is a quotient of $Y$ we have $[Y,S]^1\leq[Y,Y]^1=0$. 
We consider the short exact sequence $\xi:\xymatrix@1@C=15pt{0\ar[r]&X\ar^f[r]&Y\ar[r]&S\ar[r]&0}$. We apply the funtor $\Hom(Y,-)$ to $\xi$ and get:
$$
\xymatrix@1{0\ar[r]&\Hom(Y,X)\ar[r]&\Hom(Y,Y)\ar[r]&\Hom(Y,S)\ar[r]&\Ext^1(Y,X)=0}
$$ 
from which we deduce that $[Y,S]\leq [Y,Y]=1$ and hence $[Y,S]=1$ because $S$ is a non-zero quotient of $Y$. We then apply the functor $\Hom(-,S)$ and get  
$$
\xymatrix@1{0\ar[r]&\Hom(S,S)\ar[r]&\Hom(Y,S)\ar^F[r]&\Hom(X,S)\ar[r]&\Ext^1(S,S)\ar[r]&\Ext^1(Y,S)=0}.
$$
Since $\Hom(S,S)$ is non--zero and $\Hom(Y,S)$ is at most one--dimensional, we deduce that they are isomorphic and one--dimensional. In particular, $S$ is indecomposable. Moreover, the map $F$ is zero and hence $[S,S]^1=[X,S]$. We apply $\Hom(X,-)$ to $\xi$ and get the short exact sequence $$\xymatrix@1{0\ar[r]&\Hom(X,X)\ar[r]&\Hom(X,Y)\ar[r]&\Hom(X,S)\ar[r]&0}$$ from which we deduce  $[X,S]=[X,Y]-[X,X]=[X,Y]-1$ and the proof is complete. 
\end{proof}
\subsection{Quiver Grassmannians}\label{Sec:QuivGrass}
Let $Q$ be an acyclic quiver and $\mathbf{e}\leq\mathbf{d}$ be two dimension vectors for $Q$, where the partial order is componentwise.  Let $M\in \mathrm{R}_\mathbf{d}(Q)$ be a representation $Q$ of dimension vector $\mathbf{d}$. The quiver Grassmannian $\Gr_\mathbf{e}(M)$ parametrizes the subrepresentations of $M$ of dimension vector $\mathbf{e}$. To give the precise definition, we need to recall the construction of a family called the universal quiver Grassmannian (see \cite{Schofield}, \cite{CFR} ). We define $\Gr_\mathbf{e}(\mathbf{d}):=\prod \Gr_{e_i}(K^{d_i})$. The universal quiver Grassmannian is the incidence variety 
$$
\Gr_\mathbf{e}^Q(\mathbf{d}):=\{((U_i), (M_\alpha))\in \Gr_\mathbf{e}(\mathbf{d})\times \mathrm{R}_\mathbf{d}(Q)| M_\alpha(U_{s(\alpha)})\subseteq U_{t(\alpha)}\;\forall \alpha\in Q_1\}.
$$
It is equipped with the two projections 
%$$
%\xymatrix@R=0.5pt{&\Gr_\mathbf{e}^Q(\mathbf{d})\ar_{p_1}[dl]\ar^{p_2}[dr]&\\
%\Gr_\mathbf{e}(\mathbf{d})&&\mathrm{R}_\mathbf{d}(Q)
%}
%$$
$
\xymatrix{
\Gr_\mathbf{e}(\mathbf{d})&\Gr_\mathbf{e}^Q(\mathbf{d})\ar_{p_1}[l]\ar^{p_2}[r]&\mathrm{R}_\mathbf{d}(Q)
}
$
which are $G_\mathbf{d}$--equivariant. The map $p_1$ realizes $\Gr_\mathbf{e}^Q(\mathbf{d})$ as the total space of a homogeneous vector bundle over $\Gr_\mathbf{e}(\mathbf{d)}$. This implies that it is  irreducible, smooth of dimension $\langle\mathbf{e},\mathbf{d-e}\rangle+\textrm{dim}\,\mathrm{R}_\mathbf{d}(Q)$.  The map  $p_2$ is proper, and its image is the closed $G_\mathbf{d}$--stable subvariety of $\mathrm{R}_\mathbf{d}(Q)$ consisting  of those representations admitting a subrepresentation of dimension vector $\mathbf{e}$. The quiver Grassmannian $\Gr_\mathbf{e}(M)$ is defined to be the fiber of $p_2$ over the point $M$. In particular, if $M$ lies in the image of $p_2$, then every irreducible component of $\Gr_\mathbf{e}(M)$ has dimension at least $\langle\mathbf{e, d-e}\rangle$. In case there exists a rigid representation $M_\mathbf{d}$ of dimension vector $\mathbf{d}$, if the orbit of $M_\mathbf{d}$ lies in the image of $p_2$ then $p_2$ is surjective, and hence every representation $M\in \mathrm{R}_\mathbf{d}(Q)$ admits a subrepresentation of dimension vector $\mathbf{e}$.

We will often use the universal families $\UU$ and $\QQ$ on $\Gr_e(M)$. They arise as follows: let $\UU_i$ be the pull-back of the universal rank $e_i$ subbundle of the trivial bundle with fiber $M_i$ on $\Gr_{e_i}(M_i)$ along the natural morphism $\Gr_e(M) \to \Gr_{e_i}(M_i)$. Similarly let $\QQ_i$ the pull-back of the universal rank $d_i-e_i$ quotient bundle. The family $\UU = (\UU_i)$ satisfies $M_\alpha(\UU_i) \sub \UU_j$.  The fiber of $\mathcal{U}_i$ (resp.~$\mathcal{Q}_i$) over a subrepresentation $U$ is then canonically identified with the subspace $U_i$ of $M_i$ (resp.~the quotient $M_i/U_i$ of $M_i$).

Sometimes it will be convenient to look at the set of $T$-valued points of the quiver Grassmannian $\Gr_{\mathbf{e}}(M)$, where $T$ is a scheme over the algebraically closed field $K$. That is, we consider the contravariant functor which assigns to any $K$-scheme $T$ the set $\Hom(T,\Gr_e(M)) := \Hom_{\Spec K}(T,\Gr_e(M))$. This functor is of course uniquely determined by its values on affine $K$-schemes. Let $T = \Spec A$. Recall that the set of $T$-valued points $\Hom(T,\Gr_k(V))$ of the ordinary Grassmannian $\Gr_k(V)$ of $k$-dimensional subspaces of an $n$-dimensional $K$-vector space $V$ is in functorial bijection to the set
$$
	\{ U \mid U \text{ $A$-submodule of } V \otimes A \text{ such that } (V \otimes A)/U \text{ is projective of rank } n-k \}.
$$
The bijection is provided by the universal subbundle $\UU$ on $\Gr_k(V)$: to a morphism $f: T \to \Gr_k(V)$ we assign the pull-back $f^*\UU$ which corresponds to an $A$-module belonging to this set.
Note that, as $(V \otimes A)/U$ is projective, the short exact sequence $0 \to U \to V \otimes A \to (V \otimes A)/U \to 0$ splits and hence $U$ is also projective; its rank is $k$.

Using this universal property of the Grassmannian and definition of the quiver Grassmannian, it is not hard to see that the set of $T$-valued points of $\Gr_\mathbf{e}(M)$ is the set of tuples $(U_i)$ of submodules $U_i$ of $M_i \otimes A$ for which $(M_i \otimes A)/U_i$ is projective of rank $d_i - e_i$ and such that $(M_\alpha \otimes \id_A)(U_i) \sub U_j$ holds for all $\alpha: i \to j$.

Given a $Q$--representation $M$ we can form the representation $M^\ast=DM$ of the opposite quiver $Q^{op}$, and consider $e^\ast=\mathbf{dim}\,M-\mathbf{e}$. Then there is an isomorphism of projective varieties induced by duality of Grassmannians: $\Gr_\mathbf{e}(M)\stackrel{\cong}{\rightarrow}\Gr_{\mathbf{e}^\ast}(M^\ast): N\mapsto (M/N)^\ast=\mathrm{Ann}_{M^\ast}(N)$.
\subsection{Property (S) and cellular decomposition}	\label{Sec:PropertySandCellDec}
We recall some definitions of \cite[Sec.~1]{DLP}:
\begin{definition}\label{def:AlphaPart}
A finite partition $(X_i)$ of a complex algebraic variety $X$ is said to be an \emph{$\alpha$--partition} if 
\begin{equation}\label{Eq:DefAlphaPart}
X_1\amalg X_2\amalg\cdots\amalg X_i\textrm{ is closed in }X\textrm{ for every }i.
\end{equation}
\end{definition}
Clearly, every piece of an $\alpha$--partition is locally closed. 
\begin{definition}\label{def:CellDec}
A \emph{cellular decomposition} or \emph{affine paving} of  $X$ is an $\alpha$--partition into parts $X_i$ which are isomorphic to affine spaces.
\end{definition}
Let $X$ be a complex variety. We denote by $H_i(X) = H_i^{\operatorname{BM}}(X(\C);\Z)$ the Borel--Moore homology with integer coefficients of $X$ equipped with the analytic topology. 
\begin{definition}\label{def:PropS}
The variety $X$ has \emph{property (S)} if
	\begin{enumerate}
		\item numerical and rational equivalence on $X$ coincide,
		\item $H_i(X)=0$ for $i$ odd, and
		\item the cycle map $\varphi_i:A_i(X) \to H_{2i}(X)$ is an isomorphism for all $i$.
	\end{enumerate}
\end{definition}
(See \cite[1.3]{Fulton:98} for the definition of the Chow groups $A_i(X)$ and \cite[19.1]{Fulton:98} for the cycle map $\varphi_i$.)
Property (S) was introduced in \cite{DLP} as a replacement of cellular decomposition. Indeed, if a variety admits an $\alpha$--partition into pieces having property (S), then it has property (S) \cite[Lem.~1.8]{DLP}. In particular, cellular decomposition implies property (S). Springer fibers for classical groups admit a cellular decomposition, and for the exceptional groups have property (S) \cite[Th.~3.9]{DLP}. It is still open whether Springer fibers of exceptional groups admit a cellular decomposition \cite[Pag.~32-33]{Lusztig:Comments}.

The fact that numerical and rational equivalence on $X$ agree implies that $A_*(X)$ is a finitely generated free abelian group \cite[Ex.\ 19.1.4]{Fulton:98}. If $n$ is the complex dimension of $X$ then $A^{n-i}(X) = A_i(X)$. Supposing that $X$ is smooth, the Borel--Moore homology group $H_i(X)$ equals the singular cohomology group $H^{2n-i}(X)$. So for a smooth complex variety $X$ having property (S), the odd-degree cohomology groups $H^{2i+1}(X)$ vanish and the cycle map $A^i(X) \to H^{2i}(X)$ is an isomorphism.

The following well--known example shows that even if a variety admits a partition into affine spaces, it is not necessarily true that it admits a cellular decomposition. 
\begin{ex}
Let $X=\{[x:y:z]\in\PP^2|\, xyz=0\}$ be the union of three $\PP^1$ which  meet pairwise  in distict points. Then $X=\mathbf{A}^1\amalg \mathbf{A}^1\amalg \mathbf{A}^1$ but $H^1(X)$ is clearly non--zero and hence $X$ has no property (S). In particular, it cannot admit a cellular decomposition. 
\end{ex}

We often use the following lemma to deduce that cellular decompositions ``lift'' along affine bundles.

\begin{lem} \label{l:cell_vb}
	Let $p: A \to X$ be an affine bundle. If $X$ admits a cellular decomposition or $X$ has property (S), then so does $A$.
\end{lem}

\begin{proof}
For property (S), this is \cite[Lem.~1.9]{DLP}. 
	If $(X_i)$ is a cellular decomposition of $X$, then, by the Quillen-Suslin theorem \cite[Th.~4]{Quillen76}, $(A_i := p^{-1}(X_i))$  is a cellular decomposition of $A$. 
\end{proof}

\begin{definition}\label{Def:PropertySRepresentation}
We  say that a quiver representation has property (C) (resp. (S)) if every non--empty quiver Grassmannian associated with it admits a cellular decomposition (resp. has property (S)). 
\end{definition}

\subsection{Quasi--simple representations}
The results of this section will only be needed in Section~\ref{Sec:Preproj}, for the study of quiver Grassmannians attached to preprojective representations of affine quivers. Thus, for simplicity, throughout the section $Q$ denotes an acyclic, connected and finite quiver without multiple arrows.  In analogy with  Ringel \cite{Ringel:Wild} we give the following definition.
\begin{definition}\label{Def:QuasiSimple}
An indecomposable $Q$--representation $S$ is \emph{quasi--simple} if there are no irreducible monomorphisms ending in $S$ and there are no irreducible epimorphisms starting from $S$.
\end{definition}

\begin{prop}\label{Prop:QuasiSimple}
An indecomposable $Q$--representation $M$ is quasi--simple if and only if $M$ is either simple or an almost split sequence ending in or starting from $M$ has indecomposable middle term.
\end{prop}
\begin{proof}
Let $M$ be quasi--simple, non--projective and non--injective.  If $M$ is regular, then the claim follows from  \cite{Ringel:Wild}. By duality, we can assume that $M$ is preprojective.  Let $E_1$ and $E_2$ be two non--isomorphic indecomposable direct summands of the middle term $E$ of the almost split sequence ending in $M$. The composite map 
$\xymatrix@1{E_1\ar@{->>}[r]&M\,\ar@{^(->}[r]&\tau^-E_2}$ is neither surjective nor injective. This contradicts the Happel--Ringel Lemma~\ref{Lem:HappelRingel} and thus $E$ is indecomposable since $Q$ has no multiple arrows. 
\end{proof}
We can give another characterization of quasi--simple modules. Recall that a \emph{leaf} of $Q$ is a vertex $i\in Q_0$ which is joined to precisely one vertex by an edge of $Q$.

\begin{cor}\label{Cor:QuasiSimpleLeaves}
The $Q$--representation $M=\tau^{-t}P_{k}$ is quasi--simple if and only if $k$ is a leaf of $Q$.
\end{cor}
\begin{proof}
It follows by the construction of almost split sequences via the knitting algorithm.
\end{proof}

If $S$ is either regular or preprojective there is a simpler characterization of quasi--simplicity. 
\begin{cor}\label{Cor:QSimplePreproj}
An indecomposable, not preinjective, representation $S$ of a non--Dynkin quiver  is quasi--simple if and only if there are no irreducible monomorphisms ending in $S$. 
\end{cor}
\begin{proof}
If $S$ is regular this is well-known \cite{Ringel:Wild}. Let $S$ be preprojective without irreducible monomorphisms ending in $S$. Suppose that there is an irreducible epimorphism $\xymatrix@1{S\ar@{->>}[r]&F}$ starting from $S$. Then the composition $\xymatrix{\tau F\ar@{->>}[r]&S\ar@{->>}[r]&F}$ is epi. Hence, $\mathbf{dim}\, F>\mathbf{dim}\,\tau^{-k}F$ for every $k$, since $\tau^{-}$ preserves epis. By hypothesis $Q$ is not Dynkin, thus we get a contradiction.  
\end{proof}
\begin{lem}\label{Lem:Irred}
Let $\iota:X\hookrightarrow Y$ be an irreducible monomorphism between 
preprojective $Q$-representations and let $S=\textrm{Coker}(\iota)$. Then $S$ is indecomposable, rigid and $[S,X]^1=1$. Moreover there are no irreducible monomorphisms ending in $S$. If $S$ is either regular or preprojective then $S$ is quasi--simple. 
\end{lem}
\begin{proof}
Follows from Unger's Lemma~\ref{Lem:Unger} and corollary~\ref{Cor:QSimplePreproj}.
\end{proof}

\section{Decomposition of quiver Grassmannians induced by short exact sequences}\label{Sec:ReductionThms}
Let $Q$ be a quiver and let 
$\eta:\xymatrix{0\ar[r]&M'\ar^\iota[r]&M\ar^\pi[r]&M''\ar[r]&0}$ be a short exact sequence in $\textrm{Rep}(Q)$. Let $\mathbf{e}\in\ZZ^{Q_0}_{\geq0}$ be a dimension vector. We consider the canonical map
$$
\Psi^\eta: Gr_\mathbf{e}(M)\rightarrow\coprod_{\mathbf{f+g=e}}Gr_\mathbf{f}(M')\times Gr_\mathbf{g}(M''):\, N\mapsto(\iota^{-1}(N), \pi(N)).
$$ 
It induces subsets $
\mathcal{S}_{\mathbf{f,g}}^\eta:=(\Psi^{\eta})^{-1}(\Gr_\mathbf{f}(M')\times \Gr_\mathbf{g}(M''))
$ in $\Gr_\mathbf{e}(M)$ and a decomposition 
\begin{equation}\label{Eq:general decomposition}
\Gr_\mathbf{e}(M)=\coprod_{\mathbf{f+g=e}}\mathcal{S}_{\mathbf{f,g}}^\eta.
\end{equation}

\begin{lem}\label{Lem:AlphaPartQG}
The partition \eqref{Eq:general decomposition} is an $\alpha$--partition of $\Gr_\mathbf{e}(M)$. In particular if $\mathcal{S}_{\mathbf{f,g}}^\eta$ admits a cellular decomposition for any $\mathbf{f+g=e}$ then the same holds for  $\Gr_\mathbf{e}(M)$.
\end{lem}
\begin{proof}
Clearly the partition \eqref{Eq:general decomposition} is finite. 
The function $\Gr_\mathbf{e}(M)\rightarrow \ZZ^{Q_0}_{\geq0}: N\mapsto \textbf{dim}\,(N\cap \ker(\pi))$ which maps $N\in\Gr_\mathbf{e}(M)$ to the the dimension vector of $N\cap \ker(\pi)$ is upper--semicontinuous. In other words for every $\mathbf{n}\in\ZZ^{Q_0}_{\geq0}$ the union of  strata $\amalg_{\mathbf{f}\geq \mathbf{n}}\mathcal{S}^\eta_{\mathbf{f,g}}\subset\Gr_\mathbf{e}(M)$  is   closed in $\Gr_\mathbf{e}(M)$. By totally ordering the strata $\mathcal{S}_{\mathbf{f,g}}^\eta$ compatibly with this partial ordering, we get the statement. 
\end{proof}

For any $\mathbf{f,g}$, the  restricted map 
$
\Psi_{\mathbf{f,g}}^\eta:\mathcal{S}_{\mathbf{f,g}}^\eta\rightarrow\Gr_\mathbf{f}(M')\times\Gr_\mathbf{g}(M'')
$
is algebraic.

\begin{lem}\label{Lem:Image}
A point $(N',N'')\in\Gr_\mathbf{f}(M')\times\Gr_\mathbf{g}(M'')$ lies in the image of $\Psi_{\mathbf{f,g}}^{\eta}$ if and only if $\eta$ is in the kernel of the canonical map $\Ext^1 (M'', M') \rightarrow \Ext ^1 (N'' , M'/N')$.
\end{lem}
\begin{proof}
Given $(N', N'')$, the image $\overline{\eta}$ of $\eta$ under the composition of the canonical maps $$\xymatrix@1{\eta\in\textrm{Ext}^1 (M'', M') \ar@{->>}[r]&\Ext^1(N'',M')\ar@{->>}[r]&\Ext ^1 (N'' , M'/N')\ni\overline{\eta}}$$ corresponds to the extension $0\rightarrow M'/N'\rightarrow \overline{N}/N'\rightarrow N''\rightarrow 0$, where $\overline{N}=\pi^{-1}(N'')$.

Now $\overline{\eta}=0$ if and only if there exists a subrepresentation $N$ of $\overline{N}$, containing $N'$, projecting to $N''$ via $\pi$ and such that $M'\cap N=N'$. These are precisely the conditions defining the image of $\Psi_{\mathbf{f,g}}^{\eta}$. 
\end{proof}
Let us consider the forgetful functor $F:Rep(Q)\rightarrow Rep(Q_0)$ which associates to a $Q$--representation $M$ the $Q_0$--graded vector space $M_0:=\bigoplus_{i\in Q_0}M_i$. By definition, a quiver Grassmannian ${\rm Gr}_{\bf e}(M)$ admits a closed embedding
${\rm Gr}_{\bf e}(M)\rightarrow\prod_{i\in Q_0}{\rm Gr}_{e_i}(M_i)=\Gr_\mathbf{e}(M_0).$
We denote by $\mathcal{U}_i$ (resp.~$\mathcal{Q}_i$) the pull-back of the tautological subbundle (resp.~quotient bundle) of the trivial vector bundle $M_i$ on ${\rm Gr}_{e_i}(M_i)$ along the projection $\prod_j \Gr_{e_j}(M_j) \to \Gr_{e_i}(M_i)$. As in Section \ref{Sec:QuivGrass}, we denote the pullback of $\mathcal{U}_i$ (resp.~$\mathcal{Q}_i)$ along this embedding by the same name. The fibre of $\mathcal{U}_i$ (resp.~$\mathcal{Q}_i$) over a subrepresentation $U$ is then canonically identified with the subspace $U_i$ (resp.~the quotient $M_i/U_i$) of $M_i$. We consider the following vector bundles on $\Gr_\mathbf{f}(M'_0)\times\Gr_\mathbf{g}(M''_0)$ (see section~\ref{Sec:QuivGrass} for the definition of the various objects involved):
\begin{align*}
H_0 &:= \bigoplus_i \underline{\Hom}(p_2^*\UU''_i,p_1^*\QQ'_i), &K_0 &:= \bigoplus_{\alpha} \underline{\Hom}(p_2^*\UU''_{s(\alpha)},p_1^*\QQ'_{t(\alpha)})
\end{align*}
where $p_1:\Gr_\mathbf{f}(M'_0)\times\Gr_\mathbf{g}(M''_0)\rightarrow \Gr_\mathbf{f}(M'_0)$ and $p_2: \Gr_\mathbf{g}(M'_0)\times\Gr_\mathbf{g}(M''_0)\rightarrow \Gr_\mathbf{g}(M''_0)$ are the projections onto the two factors and $\underline{\Hom}$ denotes the vector bundle of homomorphisms. Define the bundles $H$ and $K$ on $\Gr_\mathbf{f}(M') \times \Gr_\mathbf{g}(M'')$ as the restrictions
\begin{align*}
	H &:= H_0|_{\Gr_\mathbf{f}(M') \times \Gr_\mathbf{g}(M'')}, & K &:= K_0|_{\Gr_\mathbf{f}(M') \times \Gr_\mathbf{g}(M'')}.
\end{align*}
Let $\Phi: H\rightarrow K$ be the map of vector bundles
defined over a point $(N',N'')$ as the Ringel linear map \eqref{Eq:DefiPhiQuiver}:
	$$
		\Phi_{(N',N'')}: \bigoplus_i \Hom(N''_i,M'_i/N'_i) \to \bigoplus_{\alpha:i \to j} \Hom(N''_i,M'_j/N'_j)
	$$
	which sends a tuple $f=(f_i)_i$ of linear maps $f_i: N''_i \to M'_i/N'_i$ to $\Phi_{N''}^{M'/N'}((f_i)_i) = (Q'_\alpha f_i - f_j N''_\alpha)_{\alpha:i \to j}$ (see \eqref{Eq:DefiPhiQuiver}).  In this context, $N''_\alpha: N''_i \to N''_j$ is the restriction of $M''_\alpha$ and $Q'_\alpha: M'_i/N'_i \to M'_j/N'_j$ is the induced maps by $M'_\alpha$ on the quotients. We briefly write:
$f\mapsto (f\alpha-\alpha f)_{\alpha}$.
Note that over a point $(N',N'')\in\Gr_\mathbf{f}(M')\times\Gr_\mathbf{g}(M'')$, the kernel of $\Phi_{(N',N'')}$ is $\Hom_Q(N'',M'/N')$ while its cokernel is $\Ext^1_Q(N'',M'/N')$ (see Section~\ref{Sec:Quivers}). Let $\eta_0=F(\eta)$ be the short exact sequence of $Q_0$--graded vector spaces, together with a splitting $\theta$: 
$$
\xymatrix{
\eta_0:0\ar[r]&M_0'\ar^\iota[r]&M_0\ar^\pi[r]&M_0''\ar[r]\ar@/_1pc/_{\theta}[l]&0}.
$$
The splitting $\theta$ induces a section $z_\theta=(z_{\theta,\alpha})_{\alpha\in Q_1}$ of the vector bundle $K$ defined by 
$
z_{\theta,\alpha} = \overline{\alpha \theta-\theta\alpha}
$. Here, for a morphism $g:N''\rightarrow M'$ we denote by $\overline{g}:N''\rightarrow M'/N'$ the induced morphism on the quotient. 
Note that, given any $(N',N'') \in \Gr_{\mathbf{f}}(M') \times \Gr_{\mathbf{g}}(M'')$, the image of $z_\theta(N',N'')$ under the map $K_{(N',N'')} \to \Ext^1_Q(N'',M'/N')$ agrees with the image of $\eta$ under $\Ext^1_Q(M'',M') \to \Ext^1_Q(N'',M'/N')$. 
$$
\xymatrix{
 \bigoplus_i \Hom(M''_i,M'_i) \ar[r]\ar_{\Phi_{M'}^{M''}}[d]& \bigoplus_i \Hom(N''_i,M'_i/N'_i) \ar^{\Phi_{M'/N'}^{N''}}[d]\\
 \bigoplus_{\alpha:i \to j} \Hom(M''_i,M'_j)\ar[r]\ar[d]&\bigoplus_{\alpha:i \to j} \Hom(N''_i,M'_j/N'_j)\ar[d]\\
 \Ext^1_Q(M'',M') \ar[r]& \Ext^1_Q(N'',M'/N')
 }
$$
%$$
%\xymatrix{
%K_{(M'',M')}\ar[r]\ar[d]&K_{(0,N'')}\ar[r]\ar[d]&K_{(N',N'')}\ar[d]\\
%\Ext^1(M'',M')\ar[r]&\Ext^1(N'',M')\ar[r]&\Ext^1(N'',M'/N')
%}
%$$
\begin{lem}\label{Lem:InvImageSection}
There is an isomorphism $\mathcal{S}_\mathbf{f,g}^\eta \cong \Phi^{-1}(z_\theta)$ of schemes over $\Gr_\mathbf{f}(M') \times \Gr_\mathbf{g}(M'')$.
\end{lem}

\begin{proof}
	We show that the splitting $\theta$ induces an isomorphism 
	\begin{equation}\label{Eq:IsoVectBundQ0}
		H_0 \stackrel{\cong}{\longrightarrow} \mathcal{S}_\mathbf{f,g}^{\eta_0}
	\end{equation}
	as schemes over $X_0 := \Gr_\mathbf{f}(M_0')\times\Gr_\mathbf{g}(M_0'')$. We construct this isomorphism by defining a functorial bijection between the sets $\Hom_{X_0}(T,H_0)$ and $\Hom_{X_0}(T,\smash{\mathcal{S}_{\mathbf{f},\mathbf{g}}^{\eta_0}})$. Let $T = \Spec A$ be an affine $X_0$-scheme. As explained in subsection \ref{Sec:QuivGrass} the morphism $T \to X_0$ corresponds to a pair $(N',N'')$ of $AQ_0$-submodules of $N' \sub M'_0 \otimes A$ and $N'' \sub M''_0 \otimes A$ for which $(M'_i \otimes A)/N'_i$ is a projective $A$-module of rank $\dim M'_i - f_i$ and $(M''_i \otimes A)/N''_i$ is projective of rank $\dim M''_i - g_i$. We get
	\begin{align*}
		\Hom_{X_0}(T,H_0) &= \Hom_{AQ_0}(N'',(M'_0 \otimes A)/N') \\
		\Hom_{X_0}(T,\mathcal{S}_{\mathbf{f},\mathbf{g}}^{\eta_0}) &= \{ N \sub M_0 \otimes A \mid (M_i \otimes A)/N_i \text{ proj.\ of rk. } d_i - e_i,\ \iota^{-1}(N) = N' \text{ and } \pi(N) = N'' \}.
	\end{align*}
	Here we denote the base extensions of $\iota$ and $\pi$ to maps $M'_0 \otimes A \to M_0 \otimes A \to M''_0 \otimes A$ also by $\iota$ resp.\ $\pi$.
	We send a point $f \in \Hom_{X_0}(T,H_0)$ to the subspace
	$$
		N_f = \{ \iota(m') + \theta(n'') \mid m' \in M'_0 \otimes A \text{ and } n'' \in N'' \text{ such that } m'+N' = f(n'') \}.
	$$
	A word on a piece of notation used in the above equation: by $m'+N'$ we denote the image of $m' \in M'_0 \otimes A$ under the quotient map $M'_0 \otimes A \to (M'_0 \otimes A)N'$.
	It is obvious that $\iota^{-1}(N_f) = N'$ and $\pi(N_f) = N''$, so $N_f$ lies in $\Hom_{X_0}(T,\smash{\mathcal{S}_{\mathbf{f},\mathbf{g}}^{\eta_0}})$.
	The inverse of this association is given by mapping a subspace $N$ in $\Hom_{X_0}(T,\smash{\mathcal{S}_{\mathbf{f},\mathbf{g}}^{\eta_0}})$ to the map $f_N:N'' \to (M'_0 \otimes A)/N'$ which is defined by
	$$
		f_N(n'') = m' + N'
	$$
	where $m' \in M'_0 \otimes A$ is an element for which $\iota(m') + \theta(n'') \in N$. We show that the map $f_N$ is well-defined. To show the existence of such an element $m'$ we choose to a given $n'' \in N''$ an inverse image $n \in N$ under $\pi$. Then, as $\theta$ is a splitting, there exists a unique $m' \in M'_0 \otimes A$ such that $\iota(m') + \theta(n'') = n$. Now assume $m_1', m_2' \in M'_0 \otimes A$ such that $\iota(m_1') + \theta(n'')$ and $\iota(m_2') + \theta(n'')$ both lie in $N$. Then so does $\iota(m_1'-m_2')$ and hence $m_1'-m_2' \in \iota^{-1}(N) = N'$. The map $f_N$ is hence well-defined.
	It is ease to see that these two associations are mutually inverse.
	
	Now suppose that $T \to X_0$ factors through $X = \Gr_{\mathbf{f}}(M') \times \Gr_\mathbf{g}(M'')$. That means the corresponding pair $(N',N'')$ of $AQ_0$-modules is actually a pair of $AQ$-submodules of $M' \otimes A$ and $M'' \otimes A$, respectively. 
	%As $H = H_0|X$, the sets $\Hom_X(T,H)$ and $\Hom_{X_0}(T,H_0)$ are in natural bijection. 
	An element $f \in \Hom_{X_0}(T,H_0) = \Hom_X(T,H)$ lies in $\Hom_X(T,\Phi^{-1}(z_\theta)|_X)$ if and only if $f\alpha - \alpha f = z_\theta$. A point $N \in \Hom_{X_0}(T,\smash{\mathcal{S}_{\mathbf{f},\mathbf{g}}^{\eta_0}}) = \Hom_X(T,\smash{\mathcal{S}_{\mathbf{f},\mathbf{g}}^{\eta_0}}|_X)$ is contained in $\Hom_X(T,\smash{\mathcal{S}_{\mathbf{f},\mathbf{g}}^{\eta}})$ if and only if $\alpha N \sub N$. The bijections above restrict to a bijection $\Hom_X(T,\Phi^{-1}(z_\theta)) \to \Hom_X(T,\mathcal{S}_{\mathbf{f},\mathbf{g}}^{\eta})$. We have shown that there exists an isomorphism
	$$
		\Phi^{-1}(z_\theta) \xto{}{\cong} \mathcal{S}_{\mathbf{f},\mathbf{g}}^\eta
	$$
	of schemes over $\Gr_\mathbf{f}(M') \times \Gr_\mathbf{g}(M'')$.
\end{proof}

\begin{lem}\label{Lem:Torsor}
	Let $X$ be a scheme, let $H$ and $K$ be two vector bundles on $X$, let $\phi: H \to K$ be a homomorphism of vector bundles, and let $s$ be a global section of $K$ whose image under the map $H^0(X,K) \to H^0(X,\coker(\phi))$ is zero. Then the inverse image $\phi^{-1}(s)$ has the structure of a torsor over $X$ for $\ker \phi$.
\end{lem}

\begin{proof}
	We use in the proof the abbreviations $E = \ker \phi$ and $P = \phi^{-1}(s)$. There is a cartesian square
	\begin{center}
		\begin{tikzpicture}
			\matrix(m)[matrix of math nodes,
			row sep=2.0em, column sep=2.8em,
			text height=1.5ex, text depth=0.25ex]
			{P&H\\
			X&K\\};
			\path[->,font=\scriptsize,>=angle 90]
			(m-1-1) edge (m-1-2)
			edge (m-2-1)
			(m-1-2) edge node[auto] {$\phi$} (m-2-2)
			(m-2-1) edge node[auto] {$s$} (m-2-2);
		\end{tikzpicture}
	\end{center}
	Let $T$ be a scheme over $X$.  
	Then $\Hom_X(T,P)$ is the set of all $f \in \Hom_X(T,H)$ for which
	$\phi \circ f = s|_T$.
	On the other hand $\Hom_X(T,E)$ is the subgroup of all $g \in \Hom_X(T,H)$ such that $\phi \circ g = 0$ in the abelian group $\Hom_X(T,K)$. The additive group of $\Hom_X(T,E)$ acts on $\Hom_X(T,P)$ via $(g,f) \mapsto g + f$, where the addition is taken in the abelian group $\Hom_X(T,H)$.
	To show that $P$ is a torsor for $E$ it suffices to find an open cover $\{ U_i \}_i$ of $X$ such that $\Hom_X(U_i,P) \neq \emptyset$ for every $i$. By assumption $s$ is mapped to zero under $H^0(X,K) \to H^0(X,\coker(\phi))$. That means $s$ lies in $H^0(X,\im(\phi))$. So we find an open cover $\{U_i\}_i$ of $X$ and sections $s'_i \in H^0(U_i,H)$ with $\phi \circ s'_i = s|_{U_i}$ for every $i$. Using the cartesian diagram above we get a unique morphism $r_i: U_i \to P$ such that
	\begin{center}
		\begin{tikzpicture}[descr/.style={fill=white}]
			\matrix(m)[matrix of math nodes, row sep=2em, column sep=2.8em,
			text height=1.5ex, text depth=0.25ex]
			{U_i\\&P&H\\&X&K\\};
			7
			\path[->,font=\scriptsize]
			(m-1-1) edge [bend left=25] node[auto] {$s'_i$} (m-2-3)
			(m-1-1) edge [bend right=25] (m-3-2);
			\path[->,font=\scriptsize]
			(m-1-1) edge node[auto] {$r_i$} (m-2-2);
			\path[->,font=\scriptsize]
			(m-2-2) edge (m-2-3)
			(m-2-2) edge (m-3-2);
			\path[->,font=\scriptsize]
			(m-2-3) edge node[right] {$\phi$} (m-3-3);
			\path[->,font=\scriptsize]
			(m-3-2) edge node[above] {$s$} (m-3-3);
		\end{tikzpicture}
	\end{center}
	is commutative. The map $U_i \to X$ in the above diagram is the open immersion. We have found $r_i \in \Hom_X(U_i,P)$. This implies that $P$ is a torsor for $E$.
\end{proof}

Combining Lemmas \ref{Lem:Image}, \ref{Lem:InvImageSection}, and \ref{Lem:Torsor} we are able to show

\begin{thm}\label{thm:KeyThmReduction}
	Let $Y$ be a locally closed reduced subscheme of $\Gr_\mathbf{f}(M') \times \Gr_\mathbf{g}(M'')$ which lies in the image of $\Psi_{\mathbf{f},\mathbf{g}}^\eta$. Then $(\mathcal{S}_{\mathbf{f},\mathbf{g}}^\eta)|_Y$ is a torsor for $(\ker \Phi)|_Y$. In particular, if the dimension of $\Hom_Q(N'',M'/N')$ is constant, say equal to $d$, for all closed points $(N',N'')$ of $Y$ then $(\mathcal{S}_{\mathbf{f},\mathbf{g}})|_Y$ is an affine bundle on $Y$ of rank $d$.
\end{thm}

\begin{proof}
	Consider the section $z_\theta|_Y$ of $K|_Y$ and its image $t$ in $(\coker \Phi)|_Y$. From Lemma \ref{Lem:Image} we deduce that $t(y) = 0$ for every $K$-valued point of $Y$. As $Y$ is reduced it follows that $t = 0$. We are able to apply Lemma \ref{Lem:Torsor} to $\Phi^{-1}(z_\theta|_Y)$. Together with Lemma \ref{Lem:InvImageSection}, the first claim follows. For the second claim, note that, as $Y$ is reduced, $(\ker \Phi)|_Y$ is a vector bundle if 
$[N'',M'/N']$ is constant.
\end{proof}

\subsection{Generating extensions}
In this section we analyze a class of extensions for which the fiber dimension of $\Psi_\mathbf{f,g}^\xi$ is precisely equal to $\langle\mathbf{g},\mathbf{dim}\,M'-\mathbf{f}\rangle$. We call them \emph{generating} extensions. They are defined as follows. 

\begin{definition}
Let $X$ and $S$ be two representations of the quiver $Q$. We say that an extension $\xi\in\Ext^1(S,X)$ is called \emph{generating} if it is a generator of the vector space $\Ext^1(S,X)$.
\end{definition}

By definition, a generating extension $\xi\in\Ext^1(S,X)$ is either split (in which case $[S,X]^1=0$) or it is non--split and $[S,X]^1=1$.

\begin{thm}\label{Thm:ReductionThm}
Let $\xi\in\Ext^1(S,X)$ be a generating extension. Then the image of $\Psi_\mathbf{f,g}^\xi$ is the open subset of $\Gr_\mathbf{f}(X)\times \Gr_\mathbf{g}(S)$ given by
$\textrm{Im}(\Psi_\mathbf{f,g}^\xi)=\{(N',N'')|\, [N'',X/N']^1=0\}$ and $\mathcal{S}_{\mathbf{f,g}}^\xi$ is  an affine bundle of rank $\langle\mathbf{g}, \mathbf{dim}\, X-\mathbf{f}\rangle$ over it.
\end{thm}
\begin{proof}
Let $\mathcal{Y}$ be the image of $\Psi_\mathbf{f,g}^\xi$. Since $\xi$ is generating, the description of $\mathcal{Y}$ follows immediately from Lemma~\ref{Lem:Image}. Since $\textrm{Coker }\Phi_{(N',\,N'')}\cong \Ext^1(N'',X/N')$, it follows that $\mathcal{Y}$ is precisely the locus where the map $\Phi:H\rightarrow K$ is surjective. By Lemma~\ref{Lem:InvImageSection}, $\mathcal{S}_\mathbf{f,g}^\xi\cong \Phi^{-1}(z_\theta)$ on $\mathcal{Y}$. Since $\Phi$ is surjective on $\mathcal{Y}$, for every point $y\in \mathcal{Y}$ there exists an open neighborhood $U$ such that $z_\theta|U$ is in the image of the map $H^0(U,H) \to H^0(U,K)$. We can hence use Lemma~\ref{Lem:Torsor} to deduce that, locally around every point of $\mathcal{Y}$, $\Phi^{-1}(z_\theta)$ is a trivial affine bundle for $\textrm{Ker}(\Phi)$.

\end{proof}

In case $\xi=0$ is split, then $\Psi_\mathbf{f,g}^\xi$ is surjective and $\mathcal{S}_\mathbf{f,g}^\xi$ is an affine bundle with a global section isomorphic to $\textrm{Ker}(\Phi)$. In the next section we treat the case of a non--split generating extension.

\subsection{Non--split generating extensions}
In this section we provide a more precise description of the image of the map $\Psi_\mathbf{f,g}^\xi$ for a generating extension $\xi\in\Ext^1(S,X)$ which is non--split. Throughout the section we denote the middle term of $\xi$ by $Y$. 

\begin{lem}\label{Lem:XsSXProperties}
\hspace{2cm}
\begin{enumerate}
\item
Let $N,N'\subset X$ such that $[S, X/N]^1=[S,X/N']^1=1$. Then $[S, X/(N+N')]^1=1$. 
\item Let $N,N'\subset S$ such that $[N, X]^1=[N',X]^1=1$. Then $[N\cap N', X]^1=1$. 
\end{enumerate}
\end{lem}
\begin{proof}
From $0\rightarrow \frac{X}{N\cap N'}\rightarrow \frac{X}{N}\oplus\frac{X}{N'}\rightarrow \frac{X}
{N+N'}\rightarrow 0$ we get the exact sequence
$$
\Ext^1(S, \frac{X}{N\cap N'})\rightarrow \Ext^1(S,\frac{X}{N})\oplus\Ext^1(S,\frac{X}{N'})\rightarrow \Ext^1(S, \frac{X}{N+N'})\rightarrow 0.
$$
and from $0\rightarrow N\cap N'\rightarrow N\oplus N'\rightarrow N+N'\rightarrow 0$ we get the exact sequence
$$
\Ext^1(N + N', X)\rightarrow \Ext^1(N,X)\oplus \Ext^1(N', X)\rightarrow \Ext^1(N\cap N', X)\rightarrow 0.
$$
In the above exact sequences the middle term has dimension two by hypothesis, and the extreme terms have dimension at most one. 
\end{proof}

In view of Lemma~\ref{Lem:XsSXProperties} we can give the following key definition.

\begin{definition}\label{Def:XsSx}
Let $X$ and $S$ such that $[S,X]^1=1$. We define the following subrepresentations:  
$$
\begin{array}{ccc}
X_S:=\max \{N\subset X|\, [S, X/N]^1=1\}\subset X,&\textrm{and}
&
S^X:=\min\{N\subset S|\, [N,X]^1=1\}\subseteq S.
\end{array}
$$
\end{definition}

The notation $X_S$ and $S^X$ wants to recall Ringel reflection functors \cite{Ringel:Reflection} and we call $X_S$ (resp. $S^X$) the \emph{Ringel reflection} of $X$ at $S$ (resp. of $S$ at $X$). By definition,  $X_S$ is the maximal subrepresentation of $X$ such that the push--out sequence $\pi_\ast(\xi)$ induced by the quotient $\pi: \xymatrix{X\ar@{->>}[r]&X/X_S}$ splits. Dually, $S^X$ is the minimal subrepresentation of $S$ such that the pull--back sequence $\iota^\ast(\xi)$ induced by the incusion $\iota: S^X\rightarrow S$ splits. In other words  $X_S$ is the maximal subrepresentations of $X$ such that  the quotient $\xymatrix{X\ar@{->>}[r]&X/X_S}$ factors through $Y$. Dually, $S^X$ is the minimal subrepresentation of $S$ that factors through $Y$.  In particular, if $X$ and $S$ are indecomposable such that $[S,X]^1=1$ and $\xi$ is almost split, then $X_S=0$ and $S^X=S$. The reverse implication is not true, and motivates the following definition. 

\begin{definition}\label{Def:GenAlmostSplit}
A short exact sequence $0\rightarrow X\rightarrow Y\rightarrow S\rightarrow 0$ is called \emph{generalized almost split} if  $[S,X]^1=1$, $X_S=0$ and $S^X=S$. 
\end{definition}
In other words a short exact sequence $0\rightarrow X\rightarrow Y\rightarrow S\rightarrow 0$ is generalized almost split  if $[S,X]^1=1$ and every proper quotient of $X$ and every proper subrepresentation of $S$ factors through $Y$. There are plenty of examples of generalized almost split sequences which are not almost split. 
\begin{ex}
For a quiver of type $A$, generalized almost split sequences are almost split, as one may check by direct inspection. For a quiver of type $D$ the situation is different. For example, let us consider the quiver of type $D_4$ with a unique sink in the central vertex $0$ and let $i$ be a vertex different from $0$. Put $X=P_0$ and $S=\tau^-P_i$, then $[S,X]^1=1$ and a short exact sequence corresponding to a generator $\xi\in\Ext^1(S,X)$ is generalized almost split but not almost split.  
\end{ex}
It is easy to see that if $\xi$ is generalized almost split sequence then $X$ and $S$ must be indecomposable. 

%The next lemma collects properties of $X_S$ and $S^X$. 
\begin{lem}\label{Lem:XsSX}
Let $X,S\in Rep(\mathcal{Q})$  such that $[S,X]^1=1$. 
\begin{enumerate}
\item Let $f:X\rightarrow \tau S$ be a non--zero morphism; then $X_S= \ker(f)$. 
\item Let $g:\tau^{-}X\rightarrow S$ be a non--zero morphism; then $S^X=\im(g)$. 
\item The representations $X/X_S$ and $S^X$ are indecomposable and there exist exact sequences 
\begin{eqnarray}\label{Eq:SeqSxXXs}
0\rightarrow X/X_S\rightarrow \tau S^X\rightarrow I\rightarrow 0 &&\\
\label{Eq:SeqSxXXsProj}
0\rightarrow P \rightarrow \tau^- (X/X_S) \rightarrow S^X\rightarrow 0 &&
\end{eqnarray}
where $I$ is either injective or zero and $P=\nu^{-1}(I)$.
\item The representations $X/X_S$ and $S^X$ are bricks. 
\item The vector space $\Ext^1(S^X, X/X_S)$ is  generated by a generalized almost split sequence. 
\item Given $N_1\subseteq X$ and $N_2\subseteq S$, $[N_2,X/N_1]^1=0$ if and only if either $N_1\not\subseteq X_S$ or $N_2\not\supseteq S^X$.
\end{enumerate}
\end{lem}
\begin{proof}
\begin{enumerate}
\item
We have $1=[S,X]^1=[X,\tau S]$ and let  $f:X\rightarrow \tau S$ be a non--zero morphism. We show that $\ker(f)=X_S$.  Let $p:\xymatrix{X\ar@{->>}[r]&Z}$ be a proper quotient of $X$ such that $[S,Z]^1=1$. Then $[Z,\tau S]=1$ and there is a 
non--zero composite morphism $\xymatrix{X\ar@{->>}[r]&Z\ar[r]&\tau S}$ which is hence a multiple of $f$. 
In particular, $\ker(p)\subseteq \ker(f)$. This shows that $X_S\subseteq \ker(f)$. On the other hand, $0\neq \im(f)\subseteq \tau S$ and hence $0\neq [\im(f),\tau S]=[S,\im(f)]^1\leq[S,X]^1=1$. It follows that $[S,\im(f)]^1=1$ and hence $\ker(f)\subseteq X_S$.  
\item
We have $1=[S,X]^1=[\tau^{-}X, S]$, and hence  there is a non--zero morphism $g:\tau^{-}X\rightarrow S$. Let $L\subset S$ be a subrepresentation such that $[L,X]^1=1$; then $[\tau^{-}X, L]=1$ and there is a 
non--zero composite morphism $\tau^-X\rightarrow L\subset S$ which is hence a multiple of $g$. 
It follows that $\im(g)\subseteq L$ and hence $S^X\supseteq \im(g)$. On the other hand, $\xymatrix{\tau^{-}X\ar@{->>}[r]&\im(g)}$ and hence $0\neq [\tau^-X,\im(g)]=[\im(g),X]^1\leq [S,X]^1=1$. It follows that $[\im(g),X]^1=1$ and $\im(g)\supseteq S^X$.  
\item The representation $X/X_S$ is indecomposable by maximality of $X_S$ and  $S^X$ is indecomposable by its minimality. Let us prove the existence of the sequences \eqref{Eq:SeqSxXXs} and \eqref{Eq:SeqSxXXsProj}. Let $f: X\rightarrow \tau S$ be a non--zero morphism. Then, by part (1) $X/X_S$ is the image of $f$. Let $g=\tau^{-}(f):\tau^{-}X \rightarrow S$. By part (2) $S^X$ is the image of $g$. Since $\tau$ is left exact, $\tau S^X\subseteq \tau S$ and we have a diagram
$$
\xymatrix{
X\ar[ddr]\ar@{->>}[dr]\ar^f[rr]&&\tau S\\
&X/X_S\ar@{^(->}[ur]\ar@{-->}^v[d]&\\
&\tau S^X\ar@{^(->}[uur]&
}
$$
such that the two triangles with common side f commute (since $[X,\tau S]=1$). Since $X/X_S$ is the image of $f$ there exists a unique map $v:X/X_S\rightarrow \tau S^X$ which completes the diagram, and hence must be injective. Let $I=Coker(v)$. We claim that $I$ is injective. Indeed, we apply $\tau^{-}$ to the short exact sequence $0\rightarrow X/X_S\rightarrow \tau S^X\rightarrow 
I\rightarrow 0$ and we get the exact sequence $0\rightarrow\nu^{-1}I\rightarrow \tau^{-} (X/X_S)\rightarrow S^X\rightarrow \tau^{-} I\rightarrow 0$. The morphism $\tau^{-}(v):\tau^{-} (X/X_S)\rightarrow S^X$ makes the following diagram commutative
$$
\xymatrix{
\tau^{-}X\ar@{->>}[ddr]\ar@{->>}[dr]\ar^g[rr]&& S\\
&\tau^{-}(X/X_S)\ar[ur]\ar^{\tau^{-}v}[d]&\\
&S^X\ar@{^(->}@/_1pc/[uur]&
}
$$
In particular, $\tau^{-}v$ is surjective and hence $\tau^{-}I=0$. It follows that $I$ is injective. The sequence \eqref{Eq:SeqSxXXsProj} is obtained by applying $\tau^-$ to \eqref{Eq:SeqSxXXs}.
\item We apply $\Hom(S^X,-)$ to \eqref{Eq:SeqSxXXs} and get a surjection $\xymatrix@1{\Ext^1(S^X, X/X_S)\ar@{->>}[r]&\Ext^1(S^X,\tau S^X)}$. Since $[S^X,X/X_S]^1\leq 1$ and $[S^X,\tau S^X]^1\geq1$ we see that $1=[S^X,X/X_S]^1=[S^X,\tau S^X]^1$. This implies (4), since $[S^X,S^X]=[\tau S^X,\tau S^X]=[S^X, \tau S^X]^1=1$. To get that $X/X_S$ is a brick, apply $\Hom(-, X/X_S)$ to \eqref{Eq:SeqSxXXsProj} and proceed similarly. 
\item We already proved that $1=[S^X,X/X_S]^1=[S^X,\tau S^X]^1$. Let $\eta\in\Ext^1(S^X, X/X_S)$ be non--split. Then $\eta$ has the form
$$
\xymatrix{
0\ar[r]&X/X_S\ar[r]& Y'\ar[r]&S^X\ar[r]&0
}
$$
We need to show that every (proper) quotient of $X/X_S$ and every (proper) subrepresentation of $S^X$ factor through $Y'$. This follows at once from the maximality of $X_S$ and the minimality of $S^X$: indeed, let $\xymatrix{X/X_S\ar@{->>}[r]&Z}$ be a proper quotient, then $[S,Z]^1=0$ and hence $[S^X,Z]^1\leq [S,Z]^1=0$; let $N\subsetneq S^X$, then $[N,X/X_S]^1\leq [N,X]^1=0$ by minimality of $S^X$.
\item We need to prove that $[N_2, X/N_1]^1=0$ if and only if either $[N_2,X]^1=0$ or $[S, X/N_1]^1=0$. Since $[N_2, X]^1\geq [N_2, X/N_1]^1$ and $[S,X/N_1]^1\geq [N_2,X/N_1]^1$, one implication follows. To prove the other implication let us suppose, for a contradiction, that $[N_2, X/N_1]^1=0$ and both $[N_2,X]^1=[S, X/N_1]^1=1$. Then $N_2\supseteq S^X$ and $N_1\subseteq X_S$. It follows that $[N_2, X/N_1]^1\geq [S^X, X/N_1]\geq [S^X, X/X_S]^1=1$, a contradiction. 
\end{enumerate}
\end{proof}

We can now give a more precise formulation of Theorem~\ref{Thm:ReductionThm} for a non--split generating extension.
\begin{thm}\label{Thm:Reduction2}
Let $0\neq\xi\in \Ext^1(S,X)$ be a generating extension. Then:
$$
Im (\Psi_{\mathbf{f,g}}^\xi)=(\Gr_\mathbf{f}(X)\times\Gr_\mathbf{g}(S))\setminus (\Gr_\mathbf{f}(X_S)\times\Gr_\mathbf{g-dim\,S^X}(S/S^X))
$$
and $\Psi_{\mathbf{f,g}}^\xi:\mathcal{S}_{\mathbf{f,g}}^\xi\rightarrow Im (\Psi_{\mathbf{f,g}}^\xi)$ is a Zarisky-locally trivial affine bundle of rank $\langle\mathbf{g},\mathbf{dim}\,X-\mathbf{f}\rangle$. 
\end{thm}

\begin{proof}
The description of the image follows from Lemma~\ref{Lem:XsSX}~(6). The rest is Theorem~\ref{Thm:ReductionThm}.
\end{proof}

\begin{cor}\label{Cor:RedThmAlmostSplit}
Let $N$ be a non--projective brick and let $\xi: 0\rightarrow \tau N\rightarrow E\rightarrow N\rightarrow 0$  be the almost split sequence ending in $N$. If $N$ and $\tau N$ have property (C), then $E$ has property (C). 
\end{cor}
\begin{proof}
Since $N$ is a brick, $\xi$ is generating. Since $\xi$ is almost split, it is a generalized almost split sequence. The claim now follows from Theorem~\ref{Thm:Reduction2}.
\end{proof}

The following corollary is useful for the application to cluster algebras given in Section~\ref{Sec:ClusterMultFormula}. 
	\begin{cor}\label{Cor:DecForCluster}
	Let $X$ and $S$ be two representations of $Q$ such that $[S,X]^1=1$. Let $\xi:0\rightarrow X\rightarrow Y\rightarrow S\rightarrow 0$ be a non--split short exact sequence. For any  $\mathbf{f+g=e}$ there is a decomposition
	$$
	\Gr_\mathbf{f}(X)\times \Gr_\mathbf{g}(S)= Im (\Psi_{\mathbf{f,g}}^\xi)\coprod (\Gr_\mathbf{f}(X_S)\times\Gr_\mathbf{g-dim\,S^X}(S/S^X)).
	$$
In particular the following relation between euler characteristics holds	
\begin{equation}\label{Eq:EulerCharFormula}
\sum_{\mathbf{f+g=e}} \chi(\Gr_\mathbf{f}(X)\times \Gr_\mathbf{g}(S))=\chi(\Gr_\mathbf{e}(Y))+\chi (\Gr_\mathbf{e-dim\,S^X}(X_S\oplus S/S^X))
\end{equation}

	\end{cor}

\begin{proof}
Since $[S,X]^1=1$ we have a dichotomy (see Lemma~\ref{Lem:XsSXProperties}~(6) for the second equality)
\begin{eqnarray*}
\Gr_\mathbf{f}(X)\times \Gr_\mathbf{g}(S)&=&\{(N_1,N_2)|\, [N_2, X/N_1]^1=0\}\coprod\{(N_1,N_2)|\, [N_2, X/N_1]^1=1\}\\
&=& Im (\Psi_{\mathbf{f,g}}^\xi)\coprod (\Gr_\mathbf{f}(X_S)\times \Gr_{\mathbf{g-dim}S^X}(S/S^X)).
\end{eqnarray*}
Formula~\ref{Eq:EulerCharFormula} then follows from Theorem~\ref{Thm:Reduction2}.
	\end{proof}
	
\subsection{Description of Ringel reflections in particular cases}\label{SubSec:XsSX}
In this section we investigate further $X_S$ and $S^X$ in particular cases. 
\begin{prop}\label{Prop:XsSXRigidity}
Let $X,S$ be exceptional $Q$--representations such that $[S,X]^1=1$ and $[X,S]^1=0$. Let $\xi:0\rightarrow X\rightarrow Y\rightarrow S\rightarrow 0\in\Ext^1(S,X)$ be a generating extension. Then
\begin{enumerate}
\item $S^X$ and $X/X_S$ are rigid bricks (even without the hypothesis $[X,S]^1=0$). 
\item $X\oplus Y\oplus X_S\oplus S/S^X$ is rigid. 
\item $S\oplus Y\oplus X_S\oplus S/S^X$ is rigid. 
\end{enumerate}
\end{prop}
\begin{proof}
By Lemma~\ref{Lem:XsSX}(4) we know that $S^X$ and $X/X_S$ are bricks. To show that $S^X$ is rigid apply  $\Hom(-,S)$ to the short exact sequence $0\rightarrow \textrm{Ker}(g)\rightarrow \tau^-X\rightarrow S^X\rightarrow 0$ (where $g$ is given by Lemma~\ref{Lem:XsSX}(2)) to get $[S^X,S]=1$ and $[\textrm{Ker}(g), S]=0$; then apply $\Hom(-,S^X)$ to the same sequence to get $[S^X,S^X]^1\leq[\textrm{Ker}(g), S^X]\leq [\textrm{Ker}(g), S]=0$.  
To show that $X/X_S$ is rigid, apply $\Hom(-,\tau S)$ to the short exact sequence $0\rightarrow X_S\rightarrow X\rightarrow X/X_S\rightarrow 0$ to get $[X_S,\tau S]=0$; then apply $\Hom(X,-)$ to the same sequence to get $[X/X_S,X/X_S]^1\leq [X_S,X/X_S]\leq [X_S,\tau S]=0$.

The rigidity of the two modules in (2) and (3) is obtained by constructing long exact sequences of Ext---spaces. In each case it is straightforward to find such sequences and to obtain the claimed vanishing conditions. 

\end{proof}

\begin{lem}\label{Lem:XsSXSectionalPath}
Let $X$ and $Y$ be indecomposable preprojective representations such that $[X,Y]=1$, $[X,\tau Y]=0$ and a non--zero morphism $\iota: X \rightarrow Y$ is mono. Let $S=Coker(\iota)$. Then $[S,X]^1=1$ and hence $X_S$ and $S^X$ are well--defined. Let $E\rightarrow X$ be the minimal right almost split morphism ending in $X$. If $Y$ is not projective, then there is a decompostion $E=E_0\oplus E'$ where $E_0$ is indecomposable such that $[S,E_0]^1=1$, $[S, E']^1=0$,
$
X_S\cong E'\oplus\textrm{Ker }(E_0\rightarrow \tau Y)$
and there exists a short exact sequence
$
0\rightarrow X\rightarrow \textrm{Im }(\tau^{-}E_0\rightarrow Y)\rightarrow S^X\rightarrow 0.
$
If $Y$ is projective then $X_S=E$.  
\end{lem}
\begin{proof}
From the proof of Unger's Lemma~\ref{Lem:Unger} we know that $[X,S]=0$. We apply $\Hom(-,X)$ to the short exact sequence $0\rightarrow X\rightarrow Y\rightarrow S\rightarrow 0$ and get $\mathbf{C}\cong (X,X)\rightarrow (S,X)^1\rightarrow (Y,X)^1=0$ and hence $[S,X]^1=1$. If $X$ is not--projective, let $0\rightarrow \tau X\rightarrow E\rightarrow X\rightarrow 0$ be the almost split sequence ending in $X$.  Since $[S,\tau X]^1=[\tau X, \tau S]=[X,S]=0$, it follows that $[S,E]^1=[S,X]^1=1$ and hence there exists a decomposition $E=E_0\oplus E'$ with the stated properties. Since $X$ is not projective there is a commutative diagram with exact rows
$$
\xymatrix{
0\ar[r]&\tau X\ar[r]\ar@{=}[d]&E\ar[r]\ar^\ell[d]&X\ar[r]\ar^f[d]&0\\
0\ar[r]&\tau X\ar[r]&\tau Y\ar[r]&\tau S\ar[r]&0
}
$$
It follows that $X_S=\ker(f)\cong \ker(\ell)=E'\oplus Ker (E_0\rightarrow \tau Y)$, as claimed. 
If $X=P_k$ is projective, then $E=rad(P)=\oplus_{k\rightarrow j}P_j$ and there is an exact sequence
$0\rightarrow E\rightarrow X\rightarrow S_k\rightarrow 0$.
If $Y$ is not projective, 
there exists a commutative diagram with exact rows
$$
\xymatrix{
0\ar[r]&E\ar[r]\ar^g[d]&X=P_k\ar[r]\ar[d]&S_k\ar[r]\ar@{^(->}[d]&0\\
0\ar[r]&\tau Y\ar[r]&\tau S\ar[r]&I_k\ar[r]&0\\
}
$$
from which it follows that $X_S\cong \ker(g)$. Moreover $[E,\tau Y]=[E,\tau S]=1$ and hence $E$ admits the claimed decomposition. 
If $Y=P_j$ is projective (and $X=P_k$ is projective too), then $X_S=E$ since $X/E=S_k$.  
To conclude the proof, we consider 
the commutative diagram with exact rows
$$
\xymatrix{
0\ar[r]& X\ar[r]\ar[d]&\tau^{-}E\ar[r]\ar^{\ell'}[d]&\tau^{-}X\ar[r]\ar^g[d]&0\\
0\ar[r]& X\ar[r]&Y\ar[r]& S\ar[r]&0
}
$$
In particular $Coker(\ell')\cong Coker(g)$ and hence the claimed surjection $\xymatrix{\im(\ell')\ar@{->>}[r]&\im(g)=S^X}$.
\end{proof}

\section{Proof of Theorem~\ref{t:SIntro}}\label{Sec:PropertyS}
	
	In this section we prove the following theorem. 
		\begin{thm} \label{t:S}
Let $M$ be a rigid representation of a quiver and let $X=\Gr_\mathbf{e}(M)$ be a quiver Grassmannian attached to it. Then $X$ is irreducible, has property (S) and the Chern classes of the universal bundles $\UU_i$ on $X$ generate the Chow ring $A^\ast(X)$ as a ring.	
	\end{thm}
	We start by proving the irreducibility. The main idea of the proof was communicated to the authors by Julia Sauter. 
	
	\begin{prop} \label{p:irred_rigid}
		Any quiver Grassmannian $\Gr_\mathbf{e}(M)$ attached to a rigid quiver representation $M$ is irreducible.
	\end{prop}
	
	\begin{proof}
		We consider the universal quiver Grassmannian $\Gr_\mathbf{e}^Q(\mathbf{d})\sub \Gr_\mathbf{e}(\mathbf{d}) \times \mathrm{R}_\mathbf{d}(Q)$ (see Section~\ref{Sec:QuivGrass}) and the map $p: \Gr_\mathbf{e}^Q(\mathbf{d}) \to \mathrm{R}_\mathbf{d}(Q)$ which projects to the second factor. Let $O(M)$ be the orbit of the representation $M$ by the group action of $G = \prod_i \mathrm{GL}_{d_i}(K)$ on $\mathrm{R}_\mathbf{d}(Q)$. The orbit is open in $\mathrm{R}_\mathbf{d}(Q)$ as $M$ is rigid. The inverse image $X = p^{-1}(O(M))$ is hence also open inside $\Gr_\mathbf{e}^Q(\mathbf{d})$ and thus, as $\Gr_\mathbf{e}^Q(\mathbf{d})$ is irreducible (it is an affine bundle over $\Gr_\mathbf{e}(\mathbf{d})$), $X$ is also irreducible. As we work in characteristic zero, the orbit $O(M)$ is isomorphic to $G/\Aut(M)$ via the action map. The resulting morphism $f: X \to G/\Aut(M)$ is $G$-equivariant and the fiber of $e\Aut(M)$ (the coset of the neutral element of $G$) is just $\Gr_\mathbf{e}(M)$. A general result yields that
		$$
			X \cong G \times^{\Aut(M)} \Gr_\mathbf{e}(M).
		$$
		The lemma that we are using here is well-known and the proof is very easy; this might be the reason why we could not find a reference for it other than \cite[Lem.\ 5.3]{FW:18}.
		The group $\Aut(M)$ is connected as it is open inside the vector space $\End(M)$. Now the claim follows from the following lemma.
	\end{proof}
	
	\begin{lem}\label{Lem:AlgGroup}
		Let $G$ be a connected algebraic group, let $H \sub G$ be a connected closed subgroup of $G$, let $Y$ be an $H$--scheme and $X=G\times^HY$. Then $Y_\alpha$ is an irreducible component of $Y$ if and only if $X_\alpha=G\times^HY_\alpha$  is an irreducible component of $X$. In particular, $X$ is irreducible if and only if $Y$ is irreducible. 
		\end{lem}
	
	\begin{proof}
Since $H$ is connected, each irreducible component $Y_\alpha$ is $H$--stable and   $X=\cup_\alpha X_\alpha$. The map $\varphi:G \times Y \to X$ is a principal $H$--bundle. Thus, $G \times Y_\alpha$ and hence $X_\alpha$ is irreducible. Since the map $\varphi$ is open, $X_\alpha$ is closed in  $X$, and hence an irreducible component.
	\end{proof}
	
\begin{rem}\label{Rem:IrrAnyField}
In fact proposition~\ref{p:irred_rigid} holds even without the hypothesis of characteristic zero. In positive characteristic the orbit is not necessarily isomorphic to $G/Aut(M)$, but there is a radicial homeomorphism $G/Aut(M)\rightarrow \mathcal{O}(M)$ (\cite[III.3.5.2]{DG:70}) and one applies Lemma~\ref{Lem:AlgGroup} to the pullback of the map $X\rightarrow \mathcal{O}(M)$ to conclude. 
\end{rem}
	
	A result of Ellingsrud and Str{\o}mme states:
	
	\begin{thm}[{\cite[Th.\ 2.1]{ES:93}}] \label{t:decompDiag}
		Let $X$ be a smooth complete complex variety and assume that there exist $\alpha_\lambda, \beta_\lambda \in A^*(X)$ for all $\lambda$ in a finite set $I$ such that the class of the diagonal $\Delta \sub X \times X$ decomposes as
		$$
			[\Delta] = \sum_{\lambda \in I} p_1^*\alpha_\lambda \cdot p_2^*\beta_\lambda \cap [X \times X]
		$$
		in $A_*(X \times X)$. Then the $\alpha_\lambda$'s generate $A^*(X)$ as an abelian group and $X$ has property (S).
	\end{thm}

	In the above theorem, the maps $p_1,p_2: X \times X \to X$ are the projections to the first and second component, respectively. Note that $\{ \beta_\lambda \mid \lambda \in I \}$ is also a set of generators for $A^*(X)$ as an abelian group. 
	
	\begin{rem}
		We illustrate how we can use the decomposition of the diagonal to prove that the (usual) Grassmannian $\Gr_k(V)$ has property (S). Let $V$ be an $n$-dimensional $K$-vector space. Denote by $V_X$ the trivial vector bundle on $X = \Gr_k(V)$ with fiber $V$. Let $\UU \sub V_X$ be the rank $k$--subbundle and let $\QQ = V_X/\UU$. On $X \times X$ we consider the composition of the morphisms
		$$
			p_2^*\UU \to p_2^*V_X = V_{X \times X} = p_1^*V_X \to p_1^*\QQ.
		$$
		It gives a global section $s$ of the bundle $\underline{\Hom}(p_2^*\UU,p_1^*\QQ) = p_2^*\UU^\vee \otimes p_1^*\QQ$. The zero locus $Z(s)$ agrees---as a scheme---with the diagonal $\Delta \sub X \times X$. The codimension of $Z(s)$ inside $X \times X$ is $\dim X = k(n-k)$ which is the same as the rank of $p_2^*\UU \otimes p_1^*\QQ$. Therefore, \cite[Prop.\ 14.1, Ex.\ 14.1.1]{Fulton:98} imply that
		$$
			[\Delta] = [Z(s)] = c_{k(n-k)}(p_2^*\UU^\vee \otimes p_1^*\QQ) \cap [X \times X]
		$$
		in $A_*(X \times X)$.
		If we denote by $\xi_1,\ldots,\xi_k$ the Chern roots of $\UU^\vee$ then we obtain, using a result of Lascoux (\cite{Lascoux:78}, see also \cite[Ex.\ 14.5.2]{Fulton:98}), the following expression for the top Chern class of the bundle $p_2^*\UU^\vee \otimes p_1^*\QQ$
		$$
			c_{k(n-k)}(p_2^*\UU^\vee \otimes p_1^*\QQ) = \sum_\lambda p_2^*s_\lambda(\xi_1,\ldots,\xi_k) \cdot p_1^*s_{\lambda^c}(\xi_1,\ldots,\xi_k).
		$$
		The sum ranges over all partitions $\lambda$ with $n-k \geq \lambda_1 \geq \ldots \geq \lambda_k \geq \lambda_{k+1} = 0$ and $\lambda^c$ stands for the complementary partition inside the $k \times (n-k)$ box, i.e.\ $\lambda_i + \lambda_{k-i+1}^c = n-k$ for all $i$. As the brick functions $s_\lambda$ are symmetric functions, they can be expressed as polynomials in elementary symmetric functions. Hence $s_\lambda(\xi_1,\ldots,\xi_k)$ is a polynomial in Chern classes of $\UU^\vee$ (and these lie in $A^*(X)$). The requirements of Theorem \ref{t:decompDiag} are fulfilled and thus $X$ has property (S).
	\end{rem}

	\begin{proof}[Proof of Theorem~\ref{t:S}]
	Let $\mathbf{d}=\mathbf{dim}\,M$ be the dimension vector of the rigid representation $M$ and let $\mathbf{e} \leq \mathbf{d}$ be a sub--dimension vector. We consider $X = \Gr_\mathbf{e}(M)$. It is a smooth projective variety of dimension $\langle \mathbf{e},\mathbf{d}-\mathbf{e} \rangle$. Let $\UU_i$, $\QQ_i$,  $\Phi:H\rightarrow K$ be the restrictions to $X\times X$ of the homonymous objects defined  in Section~\ref{Sec:ReductionThms}  for $M' = M'' = M$ and $\mathbf{f} = \mathbf{g} = \mathbf{e}$.  Since $M$ is rigid, the map $\Phi$ is surjective and hence $\ker(\Phi) = \underline{\Hom}_Q(p_2^*\UU,p_1^*\QQ)$ is a vector bundle of rank $\langle \mathbf{e},\mathbf{d-e} \rangle$. We use this vector bundle in the proof of Theorem \ref{t:S}. Let us consider the global section $s$ of $H$ given by
		$$
			p_2^*\UU_i \to p_2^*(M_i)_X = (M_i)_{X \times X} = p_1^*(M_i)_X \to p_1^*\QQ_i.
		$$
		This section is in fact a global section of the subbundle $\ker(\Phi)$: the fiber of the bundle $\underline{\Hom}_Q(p_2^*\UU,p_1^*\QQ)$ in a point $(U',U'')$ is $\Hom_Q(U'',M/U')$ and the composition $U'' \to M \to M/U'$ is a morphism of $K Q$-modules. The zero locus $Z(s)$ coincides with the diagonal $\Delta$ as a subscheme of $X\times X$. Then \cite[Prop.\ 14.1, Ex.\ 14.1.1]{Fulton:98} imply
		$$
			[\Delta] = i_{\Delta,*} \Z(s) = c_{\langle e,d-e \rangle}\big(\underline{\Hom}_Q(p_2^*\UU,p_1^*\QQ)\big) \cap [X \times X]
		$$
		because $X$ is smooth and the zero locus has the correct dimension. This Chern class agrees with the $\langle \mathbf{e},\mathbf{d-e} \rangle$\textsuperscript{th} coefficient of the power series
		$
			c_t(\EE) \cdot c_t(\FF)^{-1}.
		$
		Let $\xi_{i,1},\ldots,\xi_{i,\mathbf{e}_i}$ be the Chern roots of $\UU_i^\vee$. By Lascoux's result, the Chern polynomial of the bundle $\EE = \bigoplus_i \underline{\Hom}(p_2^*\UU_i,p_1^*\QQ_i)$ computes as
		$$
			c_t(\EE) = \prod_i \sum_{\lambda,\mu} \mathbf{d}_{\lambda,\mu} \cdot p_2^*s_\mu(\xi_{i,1},\ldots,\xi_{i,\mathbf{e}_i}) \cdot p_1^*s_{\lambda^c}(\xi_{i,1},\ldots,\xi_{i,\mathbf{e}_i}) \cdot t^{\mathbf{e}_i(\mathbf{d}_i-\mathbf{e}_i) + |\mu| - |\lambda|}.
		$$
		In the above equation the sum ranges over all partitions $\lambda$ and $\mu$ where $\mathbf{d}_i-\mathbf{e}_i \geq \lambda_1 \geq \ldots \geq \lambda_{\mathbf{e}_i} \geq \lambda_{\mathbf{e}_i+1} = 0$ and $\mu \sub \lambda$. The coefficient $\mathbf{d}_{\lambda,\mu}$ is the determinant of the $(\mathbf{e}_i \times \mathbf{e}_i)$-matrix whose $(k,l)$\textsuperscript{th} entry is $\binom{\lambda_k + \mathbf{e}_i - k}{\mu_l + \mathbf{e}_i - l}$. The formula for the Chern polynomial of $\FF$ is similar.
		Let $\delta(X) \sub A^*(X \times X)$ be the subset of all $\gamma$ which possess a decomposition
		$$
			\gamma = \sum_\lambda p_1^*\alpha_\lambda \cdot p_2^*\beta_\lambda
		$$
		for some classes $\alpha_\lambda,\beta_\lambda \in A^*(X)$. It is easy to see that $\delta(X)$ is a subring of $A^*(X \times X)$. We have argued that $c_t(\EE)$ and $c_t(\FF)$ lie in $\delta(X)[t]$. As the constant coefficient of $c_t(\FF)$ is one, also its inverse is a power series whose coefficients lie in $\delta(X)$. This shows that all coefficients of $c_t(\EE) \cdot c_t(\FF)^{-1}$ lie in $\delta(X)$; in particular the $\langle \mathbf{e,d-e} \rangle$\textsuperscript{th}.
		We are hence in the situation of Ellingsrud--Str{\o}mme's theorem, proving Theorem \ref{t:S}.
	\end{proof}

\begin{cor}\label{Cor:CohGrass}
 The map $\iota^\ast:{\rm H}^\bullet(\prod_{i \in Q_0}\Gr_{\mathbf{e}_i}(M_i))\rightarrow {\rm H}^\bullet(\Gr_\mathbf{e}(M))$ (induced by the closed embeddding $\iota:\Gr_\mathbf{e}(M)\rightarrow \prod_i\Gr_{\mathbf{e}_i}(M_i)$) is surjective. 
\end{cor}
\begin{proof}
By Theorem~\ref{t:S}, ${\rm H}^\bullet(\Gr_\mathbf{e}(M))$ is generated by algebraic cycles which are restrictions of universal bundles on the product of Grassmannians.  
\end{proof}
\begin{cor}\label{Cor:PropertySPreprojective}
Every $Q$--representation $M$ whose   regular part $M_\mathcal{R}$  is rigid,  has property (S).
\end{cor}
\begin{proof}
By assumption, every indecomposable direct summand of $M$ is rigid and hence the result follows from Theorem~\ref{Thm:ReductionThm} and Lemma~\ref{l:cell_vb}. 
\end{proof}

\subsection{Proof of Corollary~\ref{Cor:PolyCount}}\label{Subsec:Cor2}
Let $M$ be a rigid representation of a quiver $Q$ over an algebraically closed field $K$, let $\mathbf{d}=\mathbf{dim}\,M$ be its dimension vector and let $\Gr_\mathbf{e}(M)$ be a non--empty quiver Grassmannian attached to it. There exists a unique (up to isomorphism) rigid $\ZZ$-form $M_\ZZ\in \ZZ Q-\textrm{mod}$ such that $M\cong M_\ZZ\otimes K$, and $M_F=M_\ZZ\otimes F$ is a rigid $FQ$--module for any field $F$  \cite{CB:Rigid}. The $\ZZ$-form $M_\ZZ$ defines a 
scheme $\mathcal{X}\rightarrow \textrm{Spec}(\ZZ)$ over $\ZZ$ whose geometric fibers are $\Gr_\mathbf{e}(M_F)$.  
The scheme $\mathcal{X}$  is a  subscheme of the product $\mathcal{X}_0=\prod_i\mathcal{G}r_\ZZ(\mathbf{e}_i,\mathbf{d}_i)$ of Grassmannians over $\ZZ$ (see Section~\ref{Sec:QuivGrass}). 

\begin{lem}\label{Lem:SchemeZ}
The scheme $\mathcal{X}$ is a projective, smooth, absolutely irreducible, complete intersection subscheme of $\mathcal{X}_0$ of dimension $\langle\mathbf{e,d-e}\rangle+1$.
\end{lem}

\begin{proof}
By definition $\mathcal{X}$ is a closed subscheme of $\mathcal{X}_0$ and hence projective. Thus the map $p:\mathcal{X}\rightarrow\textrm{Spec}(\ZZ)$ is proper \cite[Th.~II.7.2]{Mumf:RedBook}.  Let $K$ be an algebraically closed field, since $M_K$ is rigid, the base change $\mathcal{X}_K$ of $\mathcal{X}$ to $K$ is smooth of dimension $\langle\mathbf{e,d-e}\rangle$, and irreducible by proposition~\ref{p:irred_rigid} and remark~\ref{Rem:IrrAnyField}. The same proof as \cite[Prop.~3.1~(iii)]{CFR} shows that $\mathcal{X}$ is complete intersection with $\sum_{i\rightarrow j\in Q_1}\mathbf{e}_i(\mathbf{d}_j-\mathbf{e}_j)$ number of equations inside the scheme $\mathcal{X}_0$  of dimension $\sum_{i\in Q_0}\mathbf{e}_i(\mathbf{d}_i-\mathbf{e}_i)+1$. So, $\langle\mathbf{e,d-e}\rangle+1$ is the dimension of the irreducible components of $\mathcal{X}$ which must all dominate $\textrm{Spec}(\ZZ)$. Thus $\mathcal{X}$ is irreducible, since the generic fiber is irreducible. Hence the morphism $p:\mathcal{X}\rightarrow\textrm{Spec}(\ZZ)$ is equidimensional in the sense of \cite[13.2.2]{EGA_IV_3}. By \cite[14.4.4, 15.2.2]{EGA_IV_3}, $p$ is flat and hence smooth.             
\end{proof}

For a prime $p$, let us denote by $K=\mathbb{Q}$ and by $\kappa=\mathbb{F}_p$. Then there is a commutative diagram:
$$
\xymatrix{
\textrm{H}^{2i}(\mathcal{X}_0(\CC),\mathbb{Q}_\ell)\ar[d]\ar[r]&
\textrm{H}^{2i}_{\textrm{\'et}}((\mathcal{X}_0)_{\CC},\mathbb{Q}_\ell)\ar[d]&
\textrm{H}^{2i}_{\textrm{\'et}}((\mathcal{X}_0)_{\overline{K}},\mathbb{Q}_\ell)\ar[d]\ar[r]\ar[l]&
\textrm{H}^{2i}_{\textrm{\'et}}((\mathcal{X}_0)_{\overline{\kappa}},\mathbb{Q}_\ell)\ar^{\iota^\ast_{\overline{\kappa}}}[d]\ar^{F^\ast}[r]&
\textrm{H}^{2i}_{\textrm{\'et}}((\mathcal{X}_0)_{\overline{\kappa}},\mathbb{Q}_\ell)\ar^{\iota^\ast_{\overline{\kappa}}}[d]\\
\textrm{H}^{2i}(\mathcal{X}(\CC),\mathbb{Q}_\ell)\ar[r]
&\textrm{H}^{2i}_{\textrm{\'et}}(\mathcal{X}_{\CC},\mathbb{Q}_\ell)
&\textrm{H}^{2i}_{\textrm{\'et}}(\mathcal{X}_{\overline{K}},\mathbb{Q}_\ell)\ar[l]\ar[r]
&\textrm{H}^{2i}_{\textrm{\'et}}(\mathcal{X}_{\overline{\kappa}},\mathbb{Q}_\ell)\ar^{F^\ast}[r]
&\textrm{H}^{2i}_{\textrm{\'et}}(\mathcal{X}_{\overline{\kappa}},\mathbb{Q}_\ell)
}
$$
where:  1) the horizontal arrows are isomorphisms (this follows from \cite[Th.~6.3~pg.~51, Cor.~3.3~pg.~63, Th.~3.1~pg.~62]{SGA_4_1_2}, passing to the limit, applied to $\mathcal{X}$ and $\mathcal{X}_0$ 
which are proper smooth over $\ZZ_{(p)}$, the localization of $\ZZ$ at $p$); 2) $F^\ast$ is induced by the geometric Frobenius morphism; 3) the leftmost vertical arrow is surjective by corollary~\ref{Cor:CohGrass}. It follows that $\iota^\ast_{\overline{\kappa}}$ is surjective. 
Moreover, since $F^\ast$ acts as multiplication by $p^i$ on $\textrm{H}^{2i}_{\textrm{\'et}}((\mathcal{X}_0)_{\overline{\kappa}},\mathbb{Q}_\ell)$ ($\mathcal{X}_0$ is a product of Grassmannians), then $F^\ast$ acts by the same scalar on  $\textrm{H}^{2i}_{\textrm{\'et}}(\mathcal{X}_{\overline{\kappa}},\mathbb{Q}_\ell)$. By the Grothendieck--Lefschetz trace formula (see e.g. \cite[(1.5.1)]{Deligne:74}):
$$
\left|\mathcal{X}(\mathbb{F}_{p^n})\right|=\sum_i\,(-1)^i\,\textrm{Tr}((F^\ast)^n, \textrm{H}^{i}_{\textrm{\'et}}(\mathcal{X}_{\overline{\kappa}},\mathbb{Q}_\ell)),
$$
calling $b_i=\textrm{dim}_\mathbb{Q}\textrm{H}^{i}(\mathcal{X}(\CC),\mathbb{Q})=\textrm{dim}_{\mathbb{Q}_\ell}\textrm{H}^{i}(\mathcal{X}(\CC),\mathbb{Q}_\ell)$ we get, for any power of a prime $q=p^n$,
$$
\left|\mathcal{X}(\mathbb{F}_{q})\right|=\sum_i b_{2i}q^i=P_\mathcal{X}(q)\in\ZZ_{\geq0}[q].
$$

\section{Proof of Theorem~\ref{t:DynkinIntro}: the Dynkin case}\label{Sec:CellDecDynkin}

\begin{thm}\label{t:Dynkin}
Every representation of a Dynkin quiver has property (C). 
\end{thm}	

\begin{proof}
Let $Q$ be a Dynkin quiver and let $M$ be a $Q$--representation. We prove that every (non--empty) quiver Grassmannian $\Gr_\mathbf{e}(M)$ attached to $M$ admits a cellular decomposition. 
The indecomposables $M(1),\cdots, M(N)$ of $\textrm{Rep}_K(Q)$ can be ordered so that $\Ext^1(M(i),M(j))=0$ if  $i>j$. Then, by Theorem~\ref{Thm:ReductionThm}, it is enough to prove Theorem~\ref{t:Dynkin} for $M$ indecomposable. We proceed by induction on the number of vertices of $Q$. The base of the induction being clear, we hence suppose the statement true for any subquiver of $Q$. In particular we may assume that $M$ is sincere, i.e. $\textrm{supp}(M)=Q$.  Let $\mathbf{d}=\mathbf{dim}\,M$. Suppose, first, that  $\mathbf{d}$ is not the longest root of $E_8$. Then $\mathbf{d}$ is minuscule, i.e. there exists a leaf $i\in Q_0$ such that $\mathbf{d}_i=1$. We may assume that $i$ is a sink. For such $i$, either $\mathbf{e}_i=0$ or $\mathbf{e}_i=1$. Let $L\subseteq M$ be the subrepresentation generated by $M_i$ (in fact it is $P_i$) and let $\overline{M}\subseteq M$ be such that $M/\overline{M}=S_i$. In both cases one sees that $\Gr_\mathbf{e}(\overline{M})$ is isomorphic to a quiver Grassmannian for a representation supported in a smaller quiver: namely  if $\mathbf{e}_i=0$ then $\Gr_\mathbf{e}(\overline{M})\stackrel{\sim}{\rightarrow} \Gr_\mathbf{e}(M)$ and
if $\mathbf{e}_i=1$ then $\Gr_{\mathbf{e-dim}\,L}(M/L)\stackrel{\sim}{\rightarrow}\Gr_\mathbf{e}(M)$ (to see this, notice that those natural maps  are bijective regular morphisms of algebraic varieties inducing isomorphisms on tangent spaces; since $M$ is rigid, $\Gr_\mathbf{e}(M)$ is smooth and hence  they are isomorphisms).  By induction we get the claim. If $\mathbf{d}$ is the longest root of $E_8$, then Lemma~\ref{Lem:DynkinLeaf} below and Corollary~\ref{Cor:RedThmAlmostSplit} ends the proof. 
\end{proof}

\begin{lem}\label{Lem:DynkinLeaf}
Let $M$ be an indecomposable of dimension vector equal to  the longest root of type $E_8$, then it is the middle term of an almost split sequence. 
\end{lem}
\begin{proof}
Let $Q$ be an arbitrary orientation of the Dynkin diagram of type $E_8$ given by
$$
\xymatrix{
&&j\ar@{-}[d]&&&&\\
\bullet\ar@{-}[r]&\bullet\ar@{-}[r]&i\ar@{-}[r]&\bullet\ar@{-}[r]&\bullet\ar@{-}[r]&\bullet\ar@{-}[r]&\bullet
}
$$
and let $\mathbf{d}=\mathbf{dim}\,M$ be the longest root:
$$
\xymatrix@C=0pt@R=0pt{
&&&3&&&&\\
\mathbf{d}:&2&4&6&5&4&3&2
}
$$
For an indecomposable $U$, consider the starting function $s_U$ which associates to any representation $V$ the number $\textrm{dim}\, \Hom_Q(U,V)$.  The starting function $s_{P_i}$ assumes the value $6$ on $M$, namely $s_{P_i}(M)=\textrm{dim}\,\Hom_Q(P_i,M)=\textrm{dim}\, M_i=6$. All starting functions for (extended) Dynkin quivers are listed in Bongartz's paper \cite{Bongartz:84}. Now by direct inspection we see that the value $6$ is only assumed in the $\tau$-orbit of $P_i$, thus M lies in that orbit. Consider the AR sequence ending in $M$, which has three middle terms. Precisely one of these middle terms, say $N$, belongs to the $\tau$-orbit of $P_j$. Thus the AR sequence ending in $N$ necessarily is of the form
$0\rightarrow N\rightarrow M\rightarrow \tau^{-}N\rightarrow 0.$
\end{proof}

\section{Proof of Theorem~\ref{t:DynkinIntro}: the affine case}\label{Sec:Affine}
In this section we prove that every indecomposable representation of an affine quiver $Q$ has property (C), in the sense that every quiver Grassmannian attached to it admits a cellular decomposition. Clearly, it is enough to deal with connected quivers. Thus, throughout the section $Q$ denotes an acyclic orientation of one of the extended Dynkin diagrams of type $\tilde{A}_n$, $\tilde{D}_n$, $\tilde{E}_6$, $\tilde{E}_7$ and $\tilde{E}_8$ depicted in table~\ref{Fig:ExtendedDynkinDiagrams}. In Section~\ref{Sec:Regular} we prove the result for the regular modules and in Section~\ref{Sec:Preproj} we prove the result for the preprojective modules. By duality we hence get the result for the preinjectives. 
\begin{table}[htbp]
\begin{center}
$$
\begin{array}{|c|c|c|}
\hline
\textrm{Type}&\textrm{Diagram}&\delta\\\hline
\hline
\xymatrix@R=10pt{\\\tilde{A}}
&
\xymatrix@C=8pt@R=8pt{
                         &                                         &\bullet\ar@{-}[drr]                                           &                                           &\\
\bullet\ar@{-}[r]\ar@{-}[urr]&\bullet\ar@{-}[r]&\cdots\ar@{-}[r]                   &\bullet\ar@{-}[r]&\bullet
                          }
                         &
\xymatrix@C=8pt@R=8pt{
                        &                                         &1\ar@{-}[drr]                                           &                                           &\\
1\ar@{-}[r]\ar@{-}[urr]&1\ar@{-}[r]&\cdots\ar@{-}[r]                   &1\ar@{-}[r]&1
                          }
\\\hline
\xymatrix@R=10pt{\\\tilde{D}}&
                          \xymatrix@C=8pt@R=5pt{
\bullet\ar@{-}[dr]&                         &                        &                                           &                                           &\bullet\\
                          &\bullet\ar@{-}[r]&\bullet\ar@{-}[r]&\cdots\ar@{-}[r]                   &\bullet\ar@{-}[ur]\ar@{-}[dr]&\\
\bullet\ar@{-}[ur]&                        &                         &                                           &                                           &\bullet
}
&
\xymatrix@C=8pt@R=5pt{
1\ar@{-}[dr]&                         &                        &                                           &                                           &1\\
                          &2\ar@{-}[r]&2\ar@{-}[r]&\cdots\ar@{-}[r]                   &2\ar@{-}[ur]\ar@{-}[dr]&\\
1\ar@{-}[ur]&                        &                         &                                           &                                           &1
}
\\\hline
\xymatrix@R=10pt{\\\tilde{E}_6}&
\xymatrix@C=10pt@R=10pt{
                        &   \bullet    \ar@{-}[r]                  &\bullet           &     &\\
\bullet\ar@{-}[r]&\bullet\ar@{-}[r]&\bullet\ar@{-}[r]\ar@{-}[u]&\bullet\ar@{-}[r]&\bullet}
&
\xymatrix@C=10pt@R=10pt{
       &               1\ar@{-}[r]        &2&&\\
1\ar@{-}[r]&2\ar@{-}[r]&3\ar@{-}[r]\ar@{-}[u]&2\ar@{-}[r]&1
}
\\\hline
\xymatrix@R=10pt{\\\tilde{E}_7}&
\xymatrix@C=10pt@R=10pt{
&& &\bullet           &  &   &\\
\bullet\ar@{-}[r]&\bullet\ar@{-}[r]&\bullet\ar@{-}[r]&\bullet\ar@{-}[r]\ar@{-}[u]&\bullet\ar@{-}[r]&\bullet\ar@{-}[r]&\bullet}
&
\xymatrix@C=10pt@R=10pt{
&& &2           &  &   &\\
1\ar@{-}[r]&2\ar@{-}[r]&3\ar@{-}[r]&4\ar@{-}[r]\ar@{-}[u]&3\ar@{-}[r]&2\ar@{-}[r]&1}
\\\hline\xymatrix@R=10pt{\\\tilde{E}_8}&
\xymatrix@C=10pt@R=10pt{
&&& &           & \bullet &   &\\
\bullet\ar@{-}[r]\ar_(.0){e}@{}[r]&\bullet\ar@{-}[r]&\bullet\ar@{-}[r]&\bullet\ar@{-}[r]&\bullet\ar@{-}[r]&\bullet\ar@{-}[u]\ar@{-}[r]&\bullet\ar@{-}[r]&\bullet}
&
\xymatrix@C=10pt@R=10pt{
&&& &           & 3 &   &\\
1\ar@{-}[r]&2\ar@{-}[r]&3\ar@{-}[r]&4\ar@{-}[r]&5\ar@{-}[r]&6\ar@{-}[u]\ar@{-}[r]&4\ar@{-}[r]&2}
\\\hline
\end{array}
$$
\caption{affine diagrams and the minimal positive imaginary roots}
\label{Fig:ExtendedDynkinDiagrams}
\end{center}
\end{table}

\subsection{Regular case}\label{Sec:Regular}
Let  $Q$ be an affine quiver shown in table~\ref{Fig:ExtendedDynkinDiagrams}. We denote by $\partial_Q:\ZZ^{Q_0}\rightarrow \ZZ$ the \emph{defect}  of $Q$ which is the linear form given by $\partial_Q(\mathbf{x}):=\langle\delta,\mathbf{x}\rangle$.
An indecomposable  $Q$--representation $M$ is preprojective, regular or preinjective  if and only if its defect $\partial_Q(M)$ is less, equal or greater than zero, respectively.  By using the defect, it is easy to see that the category $\mathcal{R}$ of regular $Q$--representations is an abelian category. This fact is not true for wild quivers (see e.g \cite{CB2}). The simple objects of $\mathcal{R}$ are precisely the quasi--simple regular $Q$--representations. Every indecomposable regular $Q$--representation $R$ admits a unique filtration: 
\begin{equation}\label{Eq:RegFilt}
R_0=0\subset R_1 \subset R_2\subset\cdots\subset R_{n-1}\subset R_{n-1}\subset R_n=R
\end{equation}
such that each $R_i$ is indecomposable regular and the successive quotients $S_k:=R_k/R_{k-1}$ are regular quasi--simple modules. Moreover the inclusions $R_{k-1}\subset R_k$ are irreducible. The quasi--simples $S_1=R_1, S_2, \cdots, S_n$ are called the regular composition factors of $R$; $R_1=S_1$ is called the regular socle of $R$ and $S_n$ is called the regular top of $R$. They satisfy the property $S_{k-1}\cong \tau S_k$ for every $k=2,\cdots, n$.  The number $n$ is called the regular length of $R$,  and $R_{n-1}$ is called the regular radical of $R$.   Every indecomposable $R$ is uniquely determined by its regular--length and its top $S$. For an indecomposable regular $R$ with regular--length $n$ and regular top $S$ we sometimes use the notations $R=R_n=R_n(S)$ and we denote by $R_i=R_i(S)$ the regular subrepresentation of regular--length $i$. 
It is worth noting that quasi-simple (indecomposable) regular representations do not have proper regular subrepresentations (since they are regular simples). In particular they are bricks. 
The restriction of the two endofunctors $\tau$ and $\tau^-$ to $\mathcal{R}$ define two inverse equivalences. The $\tau$--orbit of each indecomposable is finite and the order depends only on the regular top. The regular components of the AR--quiver are hence standard tubes. There are infinitely many tubes of rank one and at most three tubes of rank $>1$ which are called exceptional. The ranks of the exceptional tubes can be found in \cite{DlabRingel}. 

\begin{prop}\label{Prop:PropCQuasiSimpleAffineRegular}
Every quasi--simple regular $Q$--representation $S$ has property (C). Moreover, if $S$ and $T$ are regular quasi--simple such that $\mathbf{dim}\,S=\mathbf{dim}\,T=\delta$, then $\Gr_\mathbf{e}(S)\cong\Gr_\mathbf{e}(T)$ for every $\mathbf{e}$. 
\end{prop}
\begin{proof}
Let $\delta$ be the minimal positive imaginary root of $q_Q$ and let $\mathbf{d}=\mathbf{dim}\,S$.  Let $i\in Q_0$ be an extending vertex of $Q$, i.e. a vertex such that $\delta_i=1$ (see table~\ref{Fig:ExtendedDynkinDiagrams}). Up to duality, we can assume that $i$ is a source of $Q$. We show that a non--empty quiver Grassmannian $\Gr_\mathbf{e}(S)$ admits a cell decomposition. If $S$ is not supported on the whole quiver $Q$ then $S$ is supported on a Dynkin quiver and hence it has property (C) by Theorem~\ref{t:Dynkin}. We hence assume that $S$ is sincere.  It is known (see e.g. \cite[Sec.~9, Lem.~3]{CB1}) that $\mathbf{d}\leq\delta$. It follows that $\mathbf{d}_i=1$.  For every subrepresentation $N\subseteq S$ we have $[N,S/N]^1=0$ and hence $\Gr_\mathbf{e}(S)$ is smooth. We can now argue as in the proof of Theorem~\ref{t:Dynkin}: let $L(S)\subseteq S$ be the subrepresentation generated by $S_i$ and let $\overline{S}\subset S$ be such that $S/\overline{S}=S_i$. If $\mathbf{e}_i=0$ then $\Gr_\mathbf{e}(\overline{S})\stackrel{\sim}{\rightarrow} \Gr_\mathbf{e}(S)$ and if $\mathbf{e}_i=1$ then $\Gr_{\mathbf{e-dim}\,L(S)}(S/L(S))\stackrel{\sim}{\rightarrow}\Gr_\mathbf{e}(S)$ and the claim follows by Theorem~\ref{t:Dynkin}. To prove the last statement  we notice that $\overline{S}\cong \overline{T}$ and $S/L(S)\cong T/L(T)$, since they are rigid of the same dimension. 
\end{proof}

\begin{prop}\label{Prop:RegIndC}
Every indecomposable regular $Q$--representation has property (C). 
\end{prop}
\begin{proof}
Let $R=R_n(S)$ be an indecomposable regular representation of regular--lenght $n$ and regular--top $S$.  If $n=1$ then $R=S$ is quasi--simple, and the result is proved in proposition~\ref{Prop:PropCQuasiSimpleAffineRegular}. Suppose $n\geq2$. Then there is a short exact sequence 
$$
\xi_n:\xymatrix{
0\ar[r]&R_{n-1}\ar^{\iota}[r]&R_n\ar^{\pi}[r]&S\ar[r]&0}.
$$
We notice that $[S, R_{n-1}]^1=[R_{n-1}, \tau S]=1$ since $\tau S\cong R_{n-1}/R_{n-2}$ is the regular--top of $R_{n-1}$ and the quasi--simple regular representations are bricks. 
We have (see definition~\ref{Def:XsSx}) 
\begin{equation}\label{Eq:XsSXRegular}
\begin{array}{cc}
(S)^{R_{n-1}}=S,&(R_{n-1})_S=R_{n-2}.
\end{array}
\end{equation}
(This follows directly from Lemma~\ref{Lem:XsSX} parts 1 and 2.) 
We claim that each stratum $\mathcal{S}_\mathbf{f,g}^{\xi_n}$ admits a cellular decomposition for every $n\geq 2$. We proceed by induction on $n\geq2$. If $n=2$, then $\xi_2$ is almost split: by Theorem~\ref{Thm:Reduction2} we have an affine bundle
$
\xymatrix@R=2pt{
\mathcal{S}^{\xi_2} _{\mathbf{f}, \mathbf{g}}\ar@{->>}[r]&\textrm{Gr}_{\mathbf{f}}(\tau S)\times \textrm{Gr}_{\mathbf{g}}(S)
}
$
and we are done. Suppose $n\geq3$: we apply Theorem~\ref{Thm:Reduction2} to the exact sequence $\xi_n$ and in view of \eqref{Eq:XsSXRegular} we get affine bundles 
$$
\xymatrix{
\mathcal{S}^{\xi_n} _{\mathbf{f}, \mathbf{g}}\ar@{->>}[r]&\textrm{Gr}_{\mathbf{f}}(R_{n-1})\times \textrm{Gr}_{\mathbf{g}}(S)\textrm{ for }\mathbf{g}\neq \mathbf{dim}\,S,&  
\mathcal{S}^{\xi_n} _{\mathbf{f}, \mathbf{dim}\,S}\ar@{->>}[r]&\mathcal{U}_{\mathbf{f}}(R_{n-1}, R_{n-2})
}
$$
where $\mathcal{U}_{\mathbf{f}}(R_{n-1}, R_{n-2}):=\Gr_\mathbf{f}(R_{n-1})\setminus\Gr_\mathbf{f}(R_{n-2})$. By definition, 
$$
\mathcal{U}_{\mathbf{f}}(R_{n-1}, R_{n-2})=\coprod_{\mathbf{f'+f''=f},\; \mathbf{f''}\neq \mathbf{0}}\mathcal{S}_\mathbf{f',f''}^{\xi_{n-1}}.
$$
The claim follows by induction. 
\end{proof}
\begin{rem}
Formulas \eqref{Eq:XsSXRegular} can be deduced from \cite[Sec.~4]{Ringel:Wild}. Indeed Ringel proved the following: let $p:\xymatrix{R_{n-1}\ar@{->>}[r]&Z}$ be an indecomposable quotient of $R_{n-1}$. Ringel  noticed that if $Z\not\cong R_{n-1}/R_i$ for $i=0,\cdots, n-2$, then $p$ factors through $R_n$. This is another way to say that $(R_{n-1})_S=R_{n-2}$. Actually Ringel showed more: he proved that any morphism $R_{n-1}\rightarrow Z$ (not necessarily epi) factors through $R_n$ on the above hypothesis. He called this property ``extensions of homomorphisms''. 
\end{rem}

\subsection{Preprojective case}\label{Sec:Preproj}
Let $\mathcal{P}(Q)$ denote the subcategory of preprojective $Q$--representations. In this section we prove that every object of $\mathcal{P}(Q)$ has property (C).  

An indecomposable $M\in\mathcal{P}(Q)$ has the form $M=\tau^{-k}P_i$ and it has defect  $\partial (M):=\langle\delta, \mathbf{dim}\,M\rangle=-\delta_i$.
The following well--known fact will be used several times. 
\begin{lem}\label{Lem:DefectMono}
Let $M$ and $N$ be indecomposable preprojective representations of an affine quiver.  
\begin{enumerate}
\item If $|\partial(M)|\leq |\partial (N)|$ and $[N,M]^1=0$ then any non--zero morphism $f:M\rightarrow N$ is mono.
\item If there exists an epimorphism $g: \xymatrix{M\ar@{->>}[r]&N}$ then $|\partial(M)|>|\partial(N)|$.
\end{enumerate}
\end{lem}
\begin{proof}
By Happel-Ringel Lemma~\ref{Lem:HappelRingel}, $f$ is either mono or epi. If $f$ is epi, its kernel has positive defect and hence it cannot be preprojective. We conclude that $f$ must be injective.  To prove (2) we notice that the defect of the kernel of $g$, being preprojective, has negative defect. 
\end{proof}

\subsubsection{Type $\tilde{A}_n$, $\tilde{D}_n$, $\tilde{E}_6$, $\tilde{E}_7$}
Let $Q$ be an acylic orientation of an extended Dynkin diagram of one of the following types: $\tilde{A}_n$, $\tilde{D}_n$, $\tilde{E}_6$, $\tilde{E}_7$.  In this section we prove that every preprojective and every preinjective representation of $Q$ has property (C).  In order to deal with all cases at once, we introduce the  notion of a ``good monomorphism''. Recall that given two indecomposable preprojective $Q$--representations $X$ and $Y$, a sectional morphism $f:X\rightarrow Y$ is the composition
\begin{equation}\label{Eq:SectionalMorphism}
\xymatrix{
f:X=X_0\ar^(.6){f_1}[r]&X_1\ar^{f_2}[r]&X_2\ar[r]&\cdots\ar[r]&X_{t-1}\ar^{f_t}[r]&X_t=Y
}
\end{equation}
of irreducible morphisms $f_i$'s such that $X_{i-2}\not\cong\tau X_i$; in this case we say that $f$ factors through the modules $X_i$. We notice that $f$ is either mono or epi (in type $\tilde{A}$ it is mono, by Lemma~\ref{Lem:DefectMono}, in the other affine types this follows by Happel-Ringel Lemma~\ref{Lem:HappelRingel}). 
We say that $f$ is a \emph{minimal sectional mono} if it is mono and each morphism $X_{i}\rightarrow Y$ is epi for $i=1,2,\cdots, t$. In type $\tilde{A}$ a minimal sectional mono is irreducible by Lemma~\ref{Lem:DefectMono}.  An indecomposable preprojective $X=\tau^{-k}P_i$ is called a \emph{branch module} if $i\in Q_0$ is a branch vertex, i.e. it is connected to three vertices. 
\begin{definition}
Let $X$ and $Y$ be indecomposable preprojective $Q$--representations. A monomorphism $f:X\hookrightarrow Y$ is called a \emph{good mono} if 
$f$ is a sectional morphism which factors through at most one branch module. A good mono is called \emph{minimal good mono} if it is both a good mono and a minimal sectional mono. 
\end{definition}
\begin{ex}
Let $Q$ be a quiver of type $\tilde{D}_n$ and let $Y$ be an indecomposable preprojective $Q$--representation of defect $(-1)$. The section ending in $Y$ has the form: 
$$
\xymatrix@R=5pt{
Y_2 \ar[dr]&                    &          &Y_1\ar[dr]&              &\\
Y_3\ar[r]   & X_{n-3}\ar[r]&\cdots\ar[r]&X_2\ar[r]  &X_1\ar[r]&Y
}
$$ 
where $\partial(Y_i)=\partial(Y)=-1$ and $\partial(X_i)=-2$. The morphisms $Y_i \rightarrow Y$ are all sectional mono, but only the morphism $Y_1\rightarrow Y$ is good, since the other two factor through the branch modules $X_{n-3}$ and $X_1$. 
\end{ex}
\begin{rem}
We expect that there exists a more general notion of ``good morphism'' which holds for a larger class of quivers. It should be given by homological properties. The notion given here is adapted to the special choice of $Q$. The extra condition of not factoring through two branch modules only applies in type $\tilde{D}_n$ and it is needed to avoid pathological cases when $X$ is ``close to the projectives''. 
\end{rem}

We can now state the two fundamental lemmata to prove the main theorem of this section. Recall that a $Q$--representation $X$ is called \emph{thin} if $\textrm{dim}\,X_i\leq 1$ for every $i\in Q_0$. If the underlying graph of $Q$ is a tree then every indecomposable projective $Q$--representation is thin. If $Q$ is of type $\tilde{A}_n$ then all the indecomposable projective $Q$--representations are thin, except when $Q$ has precisely one source $i_0$ and precisely one sink $j_0$; in this case all projectives are thin except $P_{i_0}$ for which $\textrm{dim}\, (P_{i_0})_{j_0}=2$. Notice that $P_{i_0}$ is the only projective which is the end point of two non--isomorphic irreducible morphisms. 
\begin{lem}\label{Lem:Lemma1Preproj}
Let $Y\in\mathcal{P}(Q)$ be indecomposable. Then, either $Y$ is a quotient of an indecomposable projective $Q$--representation which is thin or there exists a good mono $f:X\hookrightarrow Y$ ending in $Y$. 
\end{lem}
\begin{proof}
Let us consider the section $\Sigma$ ending in $Y$. If $\Sigma$ is complete, 
then a look at table~\ref{Fig:ExtendedDynkinDiagrams} reveals that there exists $X\in \Sigma$ such that $|\partial(X)|\leq|\partial(Y)|$ (since $Q$ is not of type $\tilde{E}_8$) and hence the morphism $f:X\rightarrow Y$  is a sectional mono (see Lemma~\ref{Lem:DefectMono}). Thus we may choose $f$ to be a minimal sectional mono and in type $\tilde{D}_n$ such $X$ can be chosen so that the sectional morphism from $X$ to $Y$ does not factor through two branch modules. It follows that $f$ can be chosen to be a minimal good mono. 

If $\Sigma$ is not complete, then it contains a projective $X$ and the sectional map $X\rightarrow Y$ is either mono or epi by Happel-Ringel Lemma~\ref{Lem:HappelRingel}. This concludes the proof. 
\end{proof}

\begin{lem}\label{Lem:Lemma2Preproj}
Let $f:X\hookrightarrow Y$ be a minimal good mono and let $S=Coker(f)$. Then $[S,X]^1=1$ and hence $S^X$ and $X_S$ are well--defined. We have that $S^X=S$ and there are three possibilities for $X_S$: either
\begin{enumerate}
\item  $X_S=0$ or 
\item $X_S$ is indecomposable and $X_S\hookrightarrow X$ is a good mono or 
\item $X_S = F \oplus T$ with $F$ and $T$ indecomposable, $F\hookrightarrow X$ irreducible and $T\hookrightarrow X/F$  good mono.
\end{enumerate}
Moreover, $S$ is rigid, indecomposable and, if either preprojective or regular, it is quasi--simple. 
\end{lem}
\begin{proof}
If $Q$ is the Kronecker quiver, then $f$ is irreducible and $S$ is regular quasi--simple of dimension vector $\delta=(1,1)$. In particular, $[S,X]^1=[X,S]=1$. Since $f$ is irreducible,  $S^X=S$; moreover, by Lemma~\ref{Lem:XsSX} $X/X_S\cong S$ and $X_S\rightarrow X$ is irreducible. 

Let us now suppose that $Q$ is affine, neither of type $\tilde{E}_8$ nor of Kronecker type. In this case the fact that $f$ is a minimal sectional mono implies that $[X,\tau Y]=0$ and $[X,Y]=1$.  Thus, Unger's Lemma~\ref{Lem:Unger} guarantees that $S$ is indecomposable, rigid and $[S,X]^1=1$. It follows that $S^X$ and $X_S$ are well--defined (see definition~\ref{Def:XsSx}). Let us show that $S^X=S$. Let $\iota:N\hookrightarrow S$ be a proper non--zero subrepresentation of $S$. Since $f$ is a minimal sectional mono, the pullback sequence $\iota_\ast(\xi)$ splits. Since the morphism $\Ext^1(S,X)\rightarrow \Ext^1(N,X)$ is surjective and $[S,X]^1=1$, it follows that $[N,X]^1=0$. We conclude that $S^X=S$. Moreover, there are no irreducible monomorphisms ending in $S$. If $S$ is either preprojective or regular, Corollary~\ref{Cor:QSimplePreproj} hence implies that $S$ is quasi--simple.  Let us now compute $X_S$. Let $E\rightarrow X$ be the minimal right almost split map ending in $X$. Since $f:X\rightarrow Y$ is a sectional morphism, it has the form  \eqref{Eq:SectionalMorphism}. Then $E=\tau X_1\oplus E'$ for some $E'$. By Lemma~\ref{Lem:XsSXSectionalPath}, $X_S\cong \ker(\tau X_1\rightarrow \tau Y)\oplus E'$. By a case by case direct inspection, we notice that the following remarkable properties of $\mathcal{P}(Q)$ holds: 
\begin{eqnarray}\label{Fact1}
&&\textrm{The kernel of a sectional epi in $\mathcal{P}(Q)$ is a sectional mono;}\\ \label{Fact2}
&&\textrm{The cokernel of an irreducible mono $L\hookrightarrow M$ in $\mathcal{P}(Q)$ with $|\partial(L)|<|\partial(M)|$ is a sectional epi.}
\end{eqnarray}
(see remark~\ref{Rem:KernelSectional}.) We use \eqref{Fact1} and \eqref{Fact2}  to show that $X_S$ has the claimed form. 

Let us suppose first that $X$ is quasi--simple. Then $E'=0$ and $X_S=\ker(\tau X_1\rightarrow \tau Y)$. By \eqref{Fact1}, if $X_S\neq0$ then the embedding  $X_S\hookrightarrow X$ is a sectional mono. If $Q$ is of type $\tilde{D}_n$ then $|\partial(X)|=1$ and $|\partial(Y)|\in\{1,2\}$; if $|\partial(Y)|=1$, then  (since $f$ does not factor through two branch modules) $f=f_2\circ f_1$ where $f_1:X\rightarrow X_1$ and $f_2:X_1\rightarrow Y$ are irreducible, and $X_1$ is a branch module; in this case $X_S\rightarrow X$ has the form $X_S\rightarrow \tau X_1\rightarrow X$ where $X_S\rightarrow \tau X_1$ is irreducible and hence $X_S\rightarrow X$ is good mono. If $|\partial(Y)|=2$ then  $|\partial(Y)|=|\partial(X_1)|=2$ and hence $X_S=0$. 

Let us suppose that $X$ is a branch module. Then $E'=E_1\oplus E_2$ where $E_1$ and $E_2$ are indecomposable. Since $f$ does not factor through two branch modules we have 
$$
\textrm{Ker}(\tau X_1\rightarrow \tau Y)\cong \textrm{Ker}(X\rightarrow X_t)\subseteq \textrm{Ker}(X\rightarrow Y)=0.
$$
It follows that $X_S=E_1\oplus E_2$. By inspection,  $|\partial(X)|>|\partial(E_1)|$; thus \eqref{Fact2} guarantees that $X\rightarrow X/E_1$ is a sectional epi. Then $E_2\rightarrow X\rightarrow X/E_1$ is a sectional morphism. 

Let us suppose that $X$ is neither quasi--simple nor a branch module. Then $E'$ is either zero or indecomposable.  If either $\textrm{Ker}(\tau X_1\rightarrow \tau Y)=0$ or $E'=0$, then $X_S$ has the claimed form. Let us suppose that they are both non--zero. We put $F:=E'$ and $T:=\textrm{Ker}(\tau X_1\rightarrow \tau Y)$. We claim that $T\rightarrow X/F$ is a sectional morphism. In view of \eqref{Fact1}, the morphism $T\rightarrow X$ is a sectional mono and it does not factor through $F$. To prove the claim it remains to check that $X\rightarrow X/F$ is a sectional morphism.  By direct inspection, only two possibilities can occur: either $|\partial(X_1)|>|\partial(X)|>|\partial (E')|$ or $|\partial(E')|>|\partial(X)|>|\partial (X_1)|$.  Since $X_1$ is closer to the branch module of the section ending in $Y$ than $X$, again by direct inspection, $|\partial(X_1)|>|\partial(X)|>|\partial (E)|$. By \eqref{Fact2}, the cokernel $X\rightarrow X/E'$ is a sectional epi and we are done.  
\end{proof}

\begin{rem}\label{Rem:KernelSectional}
The fact that the kernel of a sectional epi in $\mathcal{P}(Q)$ is a sectional mono is not true if $Q$ is of type $\tilde{E}_8$. This is the reason why $\tilde{E}_8$ needs a slightly different treatment. 
\end{rem}
\begin{rem}
If $f:X\rightarrow Y$ is a good mono, not necessarily minimal, then the description of $X_S$ given in Lemma~\ref{Lem:Lemma2Preproj} still holds. The minimality of the morphism is needed only to get $S^X=S$. 
\end{rem}
\begin{thm}\label{Thm:Preprojective}
Every preprojective and every preinjective representation of $Q$ has property (C).
\end{thm}
\begin{proof}
Up to duality, it is enough to show the claim for preprojective $Q$--representations. Since the category $\mathcal{P}(Q)$ of preprojective $Q$--representations is directed, in view of  Theorem~\ref{Thm:ReductionThm}, it is enough to prove the claim for an indecomposable preprojective $Q$--representation. 

Let $Y$ be an indecomposable preprojective $Q$--representation. If $Y$ is the quotient of an indecomposable projective $Q$--representation which is thin, then it is thin itself. Then every quiver Grassmannian  attached to $Y$ is either empty or a point. If $Y$ is not a quotient of a projective indecomposable which is thin, then, by Lemma~\ref{Lem:Lemma1Preproj} there exists a good mono $f: X\hookrightarrow Y$ ending in $Y$. We use this monomorphism to define the open subset $\mathcal{U}_\mathbf{e}(Y,X):=\Gr_\mathbf{e}(Y)\setminus \Gr_\mathbf{e}(X)$ of $\Gr_\mathbf{e}(Y)$.  There is the obvious decomposition $\Gr_\mathbf{e}(Y)=\Gr_\mathbf{e}(X)\coprod \mathcal{U}_\mathbf{e}(Y,X)$. By induction (on $\textrm{dim}\,Y$) $X$ has property (C) and hence $\Gr_\mathbf{e}(X)$ admits a cellular decomposition. Let us show that $\mathcal{U}_\mathbf{e}(Y,X)$ admits a cellular decomposition, too. We prooceed by induction on the total dimension of $X$ and $Y$. If the sectional monomorphism $X\hookrightarrow Y$ is not minimal, then it splits as the composition of two sectional monomorphisms $X\hookrightarrow X'$ and $X'\hookrightarrow Y$ which are still good and such that $X'\hookrightarrow Y$ is minimal. We have $\mathcal{U}_\mathbf{e}(Y,X)=\mathcal{U}_\mathbf{e}(Y,X')\coprod \mathcal{U}_\mathbf{e}(X',X)$ and hence by induction $\mathcal{U}_\mathbf{e}(Y,X)$ admits a cellular decomposition. We can hence assume that $f:X\hookrightarrow Y$ is minimal. Let $S$ be its cokernel.
By Lemma~\ref{Lem:Lemma2Preproj}, $S^X=S$ and $X_S$ has one of the three forms shown there.  Moreover, $S$ is rigid and indecomposable. If $S$ is regular, then it has property (C) by proposition~\ref{Prop:PropCQuasiSimpleAffineRegular}; if $S$ is not regular, then by induction we can assume that it has property (C) (if $S$ is preinjective, its dual is preprojective and property (C) is preserved by duality).  
 By definition, $\mathcal{U}_\mathbf{e}(Y,X)=\coprod_{\mathbf{g}\neq0}\mathcal{S}_{\mathbf{f,g}}^\xi$, where $\xi\in\Ext^1(S,X)$ is the non--zero exact sequence induced by $f$. In view of  Theorem~\ref{Thm:Reduction2}, there are affine bundles 
$$
\xymatrix@R=2pt{
\mathcal{S}^{\xi} _{\mathbf{f}, \mathbf{g}}\ar@{->>}[r]&\textrm{Gr}_{\mathbf{f}}(X)\times \textrm{Gr}_{\mathbf{g}}(S)\textrm{ for }\mathbf{g}\neq \mathbf{dim}\,S,&
\mathcal{S}^{\xi} _{\mathbf{f}, \mathbf{g}}\ar@{->>}[r]&\mathcal{U}_{\mathbf{f}}(X, X_S)\textrm{ for }\mathbf{g}=\mathbf{dim}\,S.
}
$$
By induction, $\textrm{Gr}_{\mathbf{f}}(X)\times \textrm{Gr}_{\mathbf{g}}(S)$ admits a cellular decomposition, and hence $\mathcal{S}^{\xi} _{\mathbf{f}, \mathbf{g}}$ admits a cellular decomposition, too, if $\mathbf{g}\neq \mathbf{dim}\,S$. It remains to show that $\mathcal{U}_{\mathbf{f}}(X, X_S)$ admits a cellular decomposition. We consider the three possibilities. 

If $X_S=0$, then $\mathcal{U}_{\mathbf{f}}(X, X_S)=\Gr_\mathbf{f}(X)$ admits a cellular decomposition by induction.

If $X_S$ is of type (2), i.e. it is indecomposable and $X_S\hookrightarrow X$ is a good mono, then the inductive hypothesis guarantees that $\mathcal{U}_{\mathbf{f}}(X, X_S)$ admits a cellular decomposition. 

Suppose, next, that $X_S$ is of type (3), i.e. $X_S=F\oplus T\hookrightarrow X$ with $F\hookrightarrow X$ irreducible and $T\hookrightarrow X/F$ good mono. We have a commutative diagram with exact rows and columns: 
$$
\xymatrix@R=15pt{
\eta:\ar@/^10pt/@{..>}[rrrr]&&&0\ar[d]&0\ar[d]&\\
& &                     &T\ar[d] \ar@{=}[r]      &T\ar[d]&\\
\zeta:&0\ar[r]&F\ar[r]\ar@{=}[d]&X\ar^{\pi'}[r]\ar^\pi[d]&X/F\ar^{p'}[d]\ar[r]&0\\
&0\ar[r]&F\ar[r]\ar[d]&R\ar^(.3)p[r]\ar[d]        &X/(F\oplus T)\ar[r]\ar[d]&0\\
&&0&0&0&
}
$$
The middle row $\zeta$ and the righthand column $\eta$ satisfy the hypothesis of Theorem~\ref{Thm:Reduction2}. Moreover, since $F\rightarrow X$ is irreducible,   $(X/F)^F=X/F$. We can hence decompose a quiver Grassmannian $\Gr_\mathbf{e}(X)$ as $\Gr_\mathbf{e}(X)=\coprod\mathcal{S}^\zeta_{\mathbf{f,g}}$ and there are affine bundles 
$$
\xymatrix@R=2pt{
\mathcal{S}^\zeta_{\mathbf{f,g}}\ar@{->>}[r]&Gr_\mathbf{f}(F)\times Gr_\mathbf{g}(X/F)& \textrm{for } \mathbf{g}\neq \mathbf{dim}\, (X/F),
\\
\mathcal{S}^\zeta_{\mathbf{f,dim }\, X/F}\ar@{->>}[r]&\mathcal{U}_\mathbf{f}(F, F_{X/F}) &\textrm{for } \mathbf{g}= \mathbf{dim}\, (X/F).
}
$$
By definition, $\mathcal{U}_\mathbf{e}(X,F)=\coprod_{\mathbf{g}\neq 0}\mathcal{S}_{\mathbf{f,g}}^\zeta$ and  $\mathcal{U}_\mathbf{e}(X,X_S)=\coprod_{\mathbf{g}\neq 0}\mathcal{S}_{\mathbf{f,g}}^\zeta\cap \mathcal{U}_\mathbf{e}(X, X_S)$.  A point $N$ lies in $\mathcal{S}_{\mathbf{f,g}}^\zeta\cap \mathcal{U}_\mathbf{e}(X, X_S)$ if and only if $p\circ\pi(N)=p'\circ\pi'(N)\neq0$.  This means that $N\in\mathcal{S}_{\mathbf{f,g}}^\zeta\cap \mathcal{U}_\mathbf{e}(X, X_S)$ if and only if  $\pi'(N)$ belongs to $\mathcal{U}_\mathbf{g}(X/F,T)$. Thus, the affine bundles above restrict to affine bundles 
$$
\xymatrix@R=2pt{
\mathcal{S}_{\mathbf{f,g}}^\zeta\cap \mathcal{U}_\mathbf{e}(X, X_S)\ar@{->>}[r]&Gr_\mathbf{f}(F)\times \mathcal{U}_\mathbf{g}(X/F,T)& \textrm{for } 0\neq\mathbf{g}\neq  \mathbf{dim}\, (X/F),\\
\mathcal{S}^\zeta_{\mathbf{f,dim }\, X/F}=\mathcal{S}^\zeta_{\mathbf{f,dim }\, X/F}\cap \mathcal{U}_\mathbf{e}(X, X_S)\ar@{->>}[r]&\mathcal{U}_\mathbf{f}(F, F_{X/F}) &\textrm{for } \mathbf{g}= \mathbf{dim}\, (X/F).
}
$$
By induction $\Gr_\mathbf{f}(F)$, $\mathcal{U}_\mathbf{g}(X/F,T)$ and $\mathcal{U}_\mathbf{f}(F, F_{X/F})$   admit a cellular decomposition, and so does $\mathcal{S}_{\mathbf{f,g}}^\zeta\cap \mathcal{U}_\mathbf{e}(X, X_S)$ and hence $\mathcal{U}_\mathbf{e}(X, X_S)$. 

\end{proof}

\subsubsection{Type $\tilde{E}_8$}
Let  $Q$ be an affine quiver of type $\tilde{E}_8$.  In this section we prove:
\begin{thm}\label{Thm:CellDecE8Tilde}
Every  preprojective and every preinjective representation of $Q$ has property (C).
\end{thm}
Up to duality, it is enough to show the claim for preprojective $Q$--representations. Since the category $\mathcal{P}(Q)$ of preprojective $Q$--representations is directed, in view of  Theorem~\ref{Thm:ReductionThm}, it is enough to prove the claim for an indecomposable preprojective $Q$--representation. 

We first prove the claim for representations of defect (-1) (see proposition~\ref{Prop:PropCDefect1} below), then for the remaining quasi--simples (see proposition~\ref{Prop:Defect23E8} below)  and finally for all the other indecomposable representations (see proposition~\ref{Prop:NotQuasiSimpleE8}). 
Let $\delta$ be the minimal positive imaginary root of $q_Q$ (see table~\ref{Fig:ExtendedDynkinDiagrams}).  We denote by $e\in Q_0$ the extending vertex of $Q$, which is the unique vertex such that $\delta_e=1$. A representation of defect (-1) is hence of the form $\tau^{-k}P_e$ for $k\geq0$.

\begin{lem}\label{Lem:E8SES}
For every $k\geq 0$ there exists a short exact sequence 
$$
\xi_k:\xymatrix{
0\ar[r]&X:=\tau^{-k} P_e\ar[r]&Y:=\tau^{-6}X\ar[r]&S=S_k\ar[r]&0}
$$ 
such that 
\begin{enumerate}
\item $S$ is regular quasi--simple in the tube of rank $5$ and $[S,X]^1=1$.
\item If $k\geq6$, $X_S=\tau^6 X$ and $S^X=S$.
\item If $k=5$, $\xi_5$ is generalized almost split (i.e. $X_S=0$ and $S^X=S$).
\item If $k\leq4$, $X_S=0$ and $S^X=\tau^{-}X$.
\end{enumerate}
\end{lem}

\begin{table}
\begin{center}
\includegraphics[scale=0.9]{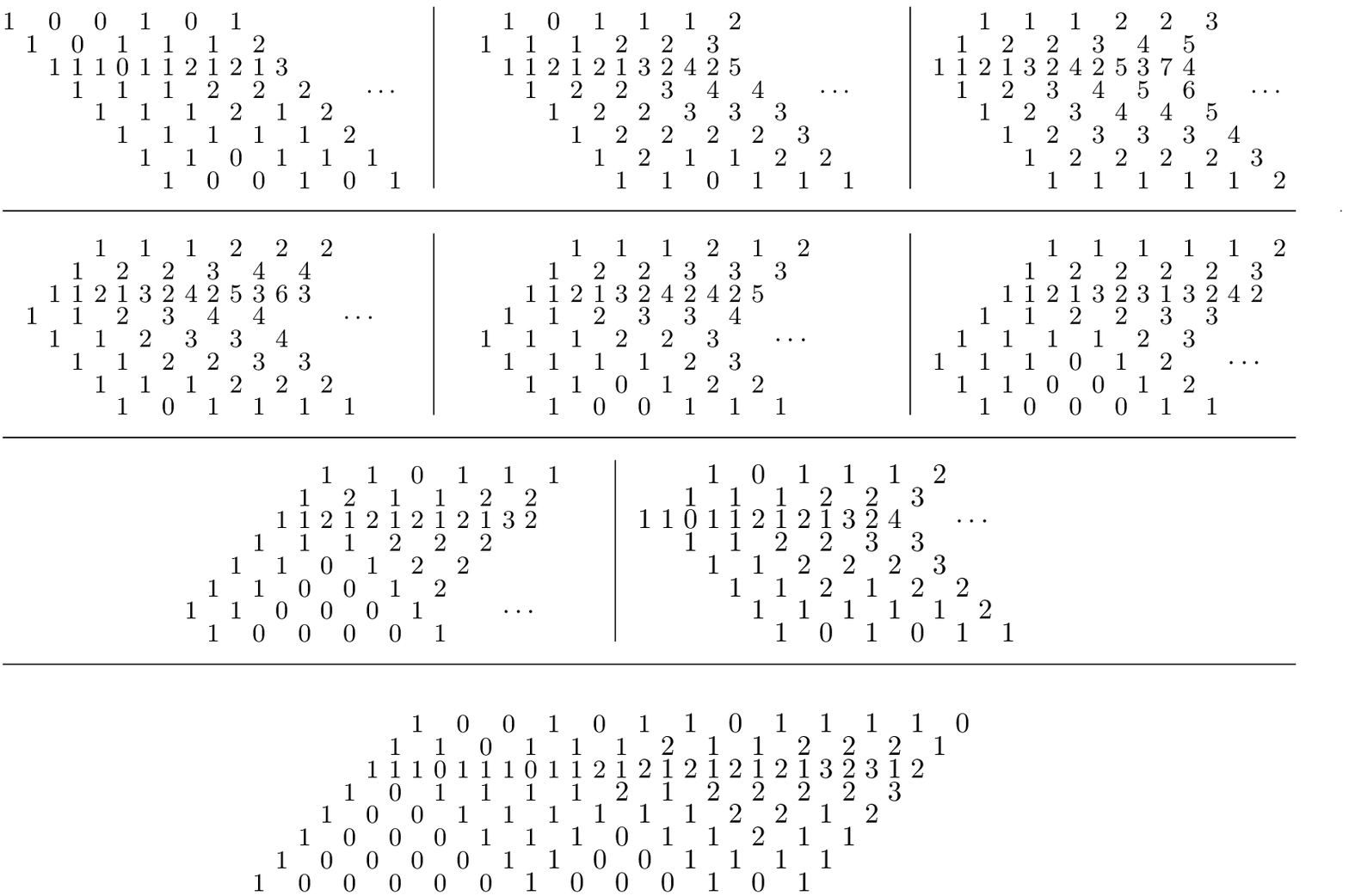}
\end{center}
\caption{The starting functions $s_{P_i}:=\textrm{dim }\Hom(P_i,-)$ for $i\in Q_0$.}\label{Fig:Bongartz}
\end{table}

\begin{proof}
In table~\ref{Fig:Bongartz} we reproduce the starting functions $s_i:=\textrm{dim }\Hom(P_i,-)$ ($i\in Q_0$) considered by Bongartz \cite{Bongartz:84}. The starting function $s_e$ is in the last row. From this we see that $[P_e, \tau^{-6}P_e]=1$ and $[P_e, \tau^{-5} P_e]=0$; thus the same holds for the $\tau^{-}$--translates and we get $[X,Y]=1$ and $[X,\tau Y]=0$. Since $X$ and $Y$ have the same defect, it follows from Lemma~\ref{Lem:DefectMono} that $X$ embeds into $Y$ and the sequence $\xi_k$ exists for every $k\geq0$. Moreover, by Unger's Lemma, we have that $[S,X]^1=1$ and that $S$ is rigid indecomposable. Since $X$ and $Y$ have the same defect, $S$ has defect zero and hence it is regular and must belong to an exceptional tube.  Since, again by table~\ref{Fig:Bongartz}, there are no $Z$ such that $X\subset Z\subset Y$ and $\partial(Z)=(-1)$, it follows that $S$ is quasi--simple.

Suppose that $k\geq 6$, so that $\tau^6X$ is non--zero. Again from Bongartz's table (last row of table~\ref{Fig:Bongartz}) we see that $[P_e,\tau^{-12}P_e]=1$ and $[P_e,\tau^{-11}P_e]=0$; thus the same holds for the $\tau^-$--translates and we get $[\tau^6 X, Y]=1$ and $[\tau^6 X, \tau Y]=0$. Thus, $\tau^6 X$ embeds both into $X$ and into $Y$ and we get an induced short exact sequence of quotients
\begin{equation}\label{Eq:AlmostSplitE8Defect1}
\xymatrix{
0\ar[r]&X/\tau^6X\ar[r]& Y/\tau^6X\ar[r]& S\ar[r]&0
}
\end{equation}
where $X/\tau^6 X$  and $S$ are indecomposable regular quasi--simple.  From Unger's lemma, we deduce that the middle term $Y/\tau^6X$ is indecomposable. This implies that  the sequence \eqref{Eq:AlmostSplitE8Defect1} is almost split and hence $X/\tau^6X\cong \tau S$. It follows that $X_S=\tau^6 X$. Moreover, since $X/X_S=\tau S$, it follows that $S^X=S$.  Let us now show that $\tau^5S_k\cong S_k$ for any $k\geq0$ and hence that $S_k$ belongs to the tube of rank 5.  Suppose that $k\geq6$. Then $\tau^6 X=X_S\neq 0$ and $\xi_{k-6}$ is given by 
$$
\xi_{k-6}:\xymatrix{
0\ar[r]&\tau^6 X\ar[r]&X\ar[r]&\tau^6 S_k\ar[r]&0}
$$ 
Since $X/\tau^6X=\tau S$, we see that $\tau^6 S_k\cong \tau S_k$ and hence $\tau^5S\cong S$. Since the exceptional tubes of a quiver of type $\tilde{E}_8$ have rank $2$, $3$ and $5$, the only possibility is that $S$ belongs to the tube of rank $5$. 

Suppose now that $k\leq 5$ so that $\tau^6 X=0$. Let us compute $X_S=\ker(X\rightarrow \tau S)$ and $S^X=\im( \tau^{-}X\rightarrow S)$. We apply the functor $\tau^6$ to $\xi_k$ and since  $\tau^6 X=0$ we get the short exact sequence
\begin{equation}\label{Eq:SESE8Injective}
\xymatrix{
0\ar[r]&\tau^6 Y=X\ar[r]&\tau^6 S=\tau S\ar[r]&\tau^{5-k}I_e\ar[r]&0.
}
\end{equation}
(Here we have used the fact that $\nu(P_e)=I_e$.) From this we see that $X$ embeds into $\tau S$ and hence $X_S=0$.  To compute $S^X$ we apply $\tau^{-}$ to \eqref{Eq:SESE8Injective} and get the short exact sequence
\begin{equation}\label{Eq:SESE8Injective2}
\xi_k': \xymatrix{
0\ar[r]&\tau^-X\ar[r]&S\ar[r]&\tau^{4-k}I_e\ar[r]&0.
}
\end{equation}
From this, we see that the map $\tau^{-}X\rightarrow S$ is injective for $k\leq 4$ and it is surjective for $k=5$. This concludes the proof. 
\end{proof} 
\begin{lem}\label{Lem:E8Def1ClosetoProj}
For $0\leq k\leq 4$ both $\tau^{-k}P_e$ and $\tau^k I_e$ are thin. A quiver Grassmannian associated with $\tau^{-5}P_e$ is either empty or a point or $\PP^1$.  In particular, $\tau^{-k}P_e$ has property (C), for every $k\in[0,5]$.
\end{lem}
\begin{proof}
The fact that $\tau^{-k}P_e$ is thin for $k\in[0,4]$ follows from table~\ref{Fig:Bongartz}. By duality, the same holds for $\tau^k I_e$.
If $X=\tau^{-5} P_e$, then again by table~\ref{Fig:Bongartz}, its dimension vector is estimated as 
\begin{equation}\label{Eq:DimTau-5Pe}
\xymatrix@1@C=0pt@R=0pt{ &&&&&&1&&\\\mathbf{dim}(\tau^{-5}P_e)\leq&0&1&1&1&1&2&1&1}
\end{equation}
Suppose that $X$ is not thin (otherwise we are done) and hence $\mathbf{dim}(X)_i=2$ where $i$ is the branch vertex.  Let us consider a non--empty quiver Grassmannian $\Gr_\mathbf{e}(X)$. If $\mathbf{e}_i\neq 1$ then  $\Gr_\mathbf{e}(X)$ is either empty or a point because it can be identified with the quiver Grassmannian of a thin representation. Suppose $\mathbf{e}_i=1$. We have a natural map $\Gr_\mathbf{e}(X)\rightarrow \mathbf{P}^1$ which is projection to vertex $i$. This is an algebraic map of projective irreducible varieties (indeed $\Gr_\mathbf{e}(X)$ is irreducible in view of proposition~\ref{p:irred_rigid}). So the image is closed and irreducible. It is hence either a point or $\PP^1$. This map is injective because outside of the branch vertex $X$ is thin and hence there is at most one way to extend a point of $\mathbf{P}^1=\mathbf{P}(X_i)$ to a subrepresentation of $X$ of dimension vector $\mathbf{e}$.  
\end{proof}

\begin{lem}\label{Lem:E8PropCDefect1}
Let $X=\tau^{-k}P_e$ for $k\in[0,5]$. Let us consider the s.e.s. 
$$
\xi_k:\xymatrix{
0\ar[r]&X\ar[r]&Y:=\tau^{-6} X\ar[r]&S\ar[r]&0}
$$ 
Then $Y$ has property (C) and $\mathcal{U}_\mathbf{e}(Y,X)$ admits a cellular decomposition for every $\mathbf{e}$.
\end{lem}
\begin{proof}
By Lemma~\ref{Lem:E8SES},  $\xi_5$ is a generalized almost split sequence and hence, since both $X$ and $S$ have property (C), then $Y$ has property (C). Moreover, $\mathcal{U}_\mathbf{e}(Y,X)=\coprod_{\mathbf{g\neq0}}\mathcal{S}^{\xi_5}_{\mathbf{f,g}}$ and there are affine bundles $\xymatrix@1{\mathcal{S}^{\xi_5}_{\mathbf{f,g}}\ar@{->>}[r]&\Gr_\mathbf{f}(X)\times\Gr_\mathbf{g}(S)}$ from which we deduce that $\mathcal{U}_\mathbf{e}(Y,X)$ admits a cellular decomposition. 

Let us now assume that $0\leq k\leq 4$. By Lemma~\ref{Lem:E8SES} we have  $[S,X]^1=1$, $X_S=0$ and $S^X=\tau^-X$. Thus there are affine bundles
$$
\xymatrix{
\mathcal{S}_{\mathbf{f,g}}^{\xi_k}\ar@{->>}[r]&\Gr_\mathbf{f}(X)\times\Gr_\mathbf{g}(S)\textrm{ for } \mathbf{f}\neq\mathbf{0},&
\mathcal{S}^{\xi_k}_{\mathbf{0,e}}\ar@{->>}[r]&\Gr_\mathbf{e}(S)\setminus\Gr_{\mathbf{e-dim}\,\tau^-X}(S/\tau^-X)\textrm{ for } \mathbf{f=0}.
}
$$
Since $X$ and $S$ have property (C), $\mathcal{S}_{\mathbf{f,g}}^{\xi_k}$ admits a cellular decomposition for $\mathbf{f}\neq\mathbf{0}$. Let us show that the same holds for $\mathcal{S}^{\xi_k}_{\mathbf{0,e}}$.
Let 
$\xi_k': \xymatrix@C=10pt{
0\ar[r]&\tau^-X\ar[r]&S\ar[r]&\tau^{4-k}I_e\ar[r]&0
}
$ 
be the short exact sequence \eqref{Eq:SESE8Injective2}. Then, by definition,  $\Gr_\mathbf{e}(S)\setminus\Gr_{\mathbf{e-dim}\,\tau^-X}(S/\tau^-X)=\coprod_{\mathbf{f\neq0}}\mathcal{S}_\mathbf{f,g}^{\xi'_k}$. Since $[\tau^{4-k}I_e, \tau^-X]^1=[\tau^{4-k}I_e, \tau^{-k-1}P_e]^1=[\tau^{-k-1}P_e, \tau^{5-k}I_e]=[P_e, \tau^6I_e]=1$, we can use  Theorem~\ref{Thm:Reduction2} to get algebraic maps which are affine bundles on their image
$$
\Psi_{\mathbf{f,g}}^{\xi_k'}:\xymatrix{
\mathcal{S}_\mathbf{f,g}^{\xi_k'}\ar@{->}[r]&\Gr_\mathbf{f}(\tau^-X)\times\Gr_\mathbf{g}(\tau^{4-k}I_e)
}
$$
If $0\leq k\leq 3$, then both $\tau^-X$ and $\tau^{4-k}I_e$ are thin by Lemma~\ref{Lem:E8Def1ClosetoProj} and hence their quiver Grassmannians are either empty or a point. It follows that the image of $\Psi_{\mathbf{f,g}}^{\xi_k'}$ is either empty or a point and hence $\mathcal{S}_\mathbf{f,g}^{\xi_k'}$ is either empty or an affine space, in this case. If $k=4$, then $\Gr_\mathbf{f}(\tau^-X)=\Gr_\mathbf{f}(\tau^{-5}P_e)$ is either empty or a point or $\PP^1$ while $\Gr_\mathbf{g}(I_e)$ is either empty or a point. The image of $\Psi_{\mathbf{f,g}}^{\xi_4'}$ is hence obtained from $\Gr_\mathbf{f}(\tau^-X)$ by removing $\Gr_\mathbf{f}((\tau^-X)_{I_e})$. In view of proposition~\ref{Prop:XsSXRigidity} (part (2)), $(\tau^-X)_{I_e}$ is rigid and hence, by proposition~\ref{p:irred_rigid}, $\Gr_\mathbf{f}((\tau^-X)_{I_e})$ is irreducible. It follows that if $\Gr_\mathbf{f}(\tau^-X)$ is isomoprhic to $\PP^1$, $\Gr_\mathbf{f}((\tau^-X)_{I_e})$ is either empty or a point or $\PP^1$. In all cases, we see that the image of $\Psi_{\mathbf{f,g}}^{\xi_4'}$ is an affine space, and hence $\mathcal{S}_\mathbf{f,g}^{\xi_4'}$ is an affine space, too.
\end{proof}

\begin{prop}\label{Prop:PropCDefect1}
If $Y=\tau^{-k}P_e\in\mathcal{P}(Q)$ has defect $(-1)$, then $Y$ has property (C) and, for $k\geq 6$, both $\mathcal{U}_\mathbf{e}(Y,\tau^6Y)$ and $\Gr_\mathbf{e}(Y)\setminus(\Gr_{\mathbf{e-dim}\,\tau^6Y})(Y/\tau^6Y)$ admit a  cellular decomposition, for every $\mathbf{e}$. 
\end{prop}
\begin{proof}
We proceed by induction on $k\geq 0$. If $k\in[0,5]$ then $Y$ has property (C) by Lemma~\ref{Lem:E8Def1ClosetoProj}. 
If $6\leq k\leq 11$ then, by proposition~\ref{Lem:E8PropCDefect1}, $Y$ has property (C) and $\mathcal{U}_\mathbf{e}(Y,\tau^6Y)$ admits a cellular decomposition.
If $k\geq12$, then there exits a short exact sequence $\xi=\xi_{k-6}: 0\rightarrow \tau^6Y\rightarrow Y\rightarrow S\rightarrow 0$ with the good properties hilighted in Lemma~\ref{Lem:E8SES}. There are affine bundles 
$$
\xymatrix@R=2pt{
\mathcal{S}^{\xi} _{\mathbf{f}, \mathbf{g}}\ar@{->>}[r]&\textrm{Gr}_{\mathbf{f}}(\tau^6Y)\times \textrm{Gr}_{\mathbf{g}}(S)\textrm{ for }\mathbf{g}\neq \mathbf{dim}\,S,&
\mathcal{S}^{\xi} _{\mathbf{f}, \mathbf{g}}\ar@{->>}[r]&\mathcal{U}_{\mathbf{f}}(\tau^6Y, \tau^{12} Y)\textrm{ for }\mathbf{g}=\mathbf{dim}\,S.
}
$$
By induction and proposition~\ref{Prop:PropCQuasiSimpleAffineRegular}, both $\textrm{Gr}_{\mathbf{f}}(\tau^6Y)\times \textrm{Gr}_{\mathbf{g}}(S)$ and $\mathcal{U}_{\mathbf{f}}(\tau^6Y, \tau^{12} Y)$ admit a cellular decomposition and hence each stratum $\mathcal{S}^{\xi} _{\mathbf{f}, \mathbf{g}}$ admits a cellular decomposition, too.  By definition,   $\Gr_\mathbf{e}(Y)\setminus(\Gr_\mathbf{e-dim}\,\tau^6Y)(Y/\tau^6Y)=\coprod_{\mathbf{f}\neq\mathbf{0}}\mathcal{S}_\mathbf{f,g}^\xi$ and hence it admits a cellular decomposition. 
\end{proof}

Let us now deal with the remaining quasi--simple representations. 
\begin{prop}\label{Prop:Defect23E8}
Every quasi--simple preprojective $Q$--representation has property (C). 
\end{prop}
\begin{proof}
By Proposition~\ref{Prop:PropCDefect1}  the claim holds for the quasi--simple representations of defect $(-1)$. Let us now deal with the remaining ones, which have defect either (-2) or (-3). 

Let $X\in\mathcal{P}(Q)$ be quasi--simple of defect $(-2)$ and let $Y\in\mathcal{P}(Q)$ be the quasi--simple of defect $(-1)$ such that $X$ lies in the section ending in $Y$ (see figure~\ref{Fig:PreprojCompE8}). By direct check, one verifies that the kernel of the sectional morphism $X\rightarrow Y$ is the the sectional morphism $\tau^7Y\rightarrow X$.  If $\tau^7Y=0$ then $\tau^6Y$ is either projective or zero; in both cases $[P_e,X]=0$ and hence $X$ is supported on a Dynkin quiver of type $E_8$. Thus, by Theorem~\ref{t:Dynkin}, $X$ has property (C). Let us hence suppose that $\tau^7Y\neq 0$.  Then we consider the short exact sequence 
$$
\xi:\xymatrix{
0\ar[r]&\tau^7 Y \ar[r]&X\ar[r]&Y\ar[r]&0.
}
$$
Since $[Y,\tau^7 Y]^1=[\tau^7 Y, \tau Y]=[\tau^6 Y, Y]=1$, we see that $(\tau^7 Y)_Y$ and $Y^{(\tau^7Y)}$ are well--defined (see Definition~\ref{Def:XsSx}). For defect reasons, $(\tau^7 Y)_{Y}=\textrm{Ker}(\tau^7 Y\rightarrow Y)=0$ and $(Y)^{\tau^7 Y}=\textrm{Im}(\tau^6Y\rightarrow Y)=\tau^6 Y$. By  Theorem~\ref{Thm:Reduction2}, there are (Zariski--locally trivial) affine bundles
$$
\xymatrix@R=2pt{
\mathcal{S}^{\xi} _{\mathbf{f}, \mathbf{g}}\ar@{->>}[r]&\textrm{Gr}_{\mathbf{f}}(\tau^7Y)\times \textrm{Gr}_{\mathbf{g}}(Y)\textrm{ for }\mathbf{f}\neq \mathbf{0},&
\mathcal{S}^{\xi} _{\mathbf{0}, \mathbf{e}}\ar@{->>}[r]&\Gr_\mathbf{e}(Y)\setminus \Gr_{\mathbf{e-dim}\,\tau^6Y}(Y/\tau^6Y)\textrm{ for }\mathbf{f}=\mathbf{0}.
}
$$
By proposition~\ref{Prop:PropCDefect1}, both $\textrm{Gr}_{\mathbf{f}}(\tau^7Y)\times \textrm{Gr}_{\mathbf{g}}(Y)$ and $\Gr_\mathbf{e}(Y)\setminus \Gr_{\mathbf{e-dim}\,\tau^6Y}(Y/\tau^6Y)$ admit a cellular decomposition and hence each stratum $\mathcal{S}^{\xi} _{\mathbf{f}, \mathbf{g}}$ has the same property, proving property (C) for $X$.

Let $Z\in\mathcal{P}(Q)$ be quasi--simple of defect $(-3)$ and let $X'\in\mathcal{P}(Q)$ be the quasi--simple of defect $(-2)$ such that $Z$ lies in the section ending in $X'$ (see figure~\ref{Fig:PreprojCompE8}).  By direct check, one verifies that the kernel of the sectional morphism $Z\rightarrow X'$ is the sectional morphism $Y'\rightarrow Z$, where $Y'$ is the quasi--simple of defect $(-1)$ in the section ending in $Z$. If $Y'=0$ then $[P_e,Z]=0$ and hence $Z$ is supported on a quiver of type $E_8$ and by Theorem~\ref{t:Dynkin}, $Z$ has property (C). Let us hence suppose that $Y'\neq0$. Then there is a short exact sequence 
$$
\eta:\xymatrix{
0\ar[r]&Y'\ar[r]&Z\ar[r]&X'\ar[r]&0.
}
$$
We have $[X',Y']^1=[Y', \tau X']=1$ and hence $(Y')_{X'}$ and $X'^{(Y')}$ are well--defined (see Definition~\ref{Def:XsSx}). For defect reasons, the sectional morphisms $Y'\rightarrow \tau X'$ and $\tau^-Y'\rightarrow X'$ are monomorphisms, and hence $(Y')_{X'}=0$ and $X'^{(Y')}=\tau^-Y'$.  It follows that there are (Zarisky-locally trivial) affine bundles 
$$
\xymatrix@C=20pt{
\mathcal{S}^{\eta} _{\mathbf{f}, \mathbf{g}}\ar@{->>}[r]&\textrm{Gr}_{\mathbf{f}}(Y')\times \textrm{Gr}_{\mathbf{g}}(X')\textrm{ for }\mathbf{f}\neq \mathbf{0},&
\mathcal{S}^{\eta} _{\mathbf{0}, \mathbf{e}}\ar@{->>}[r]&\Gr_\mathbf{e}(X')\setminus \Gr_{\mathbf{e-dim}\,\tau^-Y'}(X'/\tau^-Y')\textrm{ for }\mathbf{f}=\mathbf{0}.
}
$$
We have already proved that $\textrm{Gr}_{\mathbf{f}}(Y')\times \textrm{Gr}_{\mathbf{g}}(X')$ admits a cellular decomposition. To get a cellular decomposition of $\Gr_\mathbf{e}(X')\setminus \Gr_{\mathbf{e-dim}\,\tau^-Y'}(X'/\tau^-Y')$ we consider the short exact sequence 
$$
\xi':\xymatrix{
0\ar[r]&\tau^- Y' \ar[r]&X'\ar[r]&\tau^{-8}Y'\ar[r]&0.
}
$$
We proved above that each stratum $\mathcal{S}_\mathbf{f,g}^{\xi'}$ admits a cellular decomposition. It follows that $\Gr_\mathbf{e}(X')\setminus \Gr_{\mathbf{e-dim}\,\tau^-Y'}(X'/\tau^-Y')=\coprod_{\mathbf{f}\neq0}\mathcal{S}_\mathbf{f,g}^{\xi'}$ admits a cellular decomposition, too. This concludes the proof. 
\end{proof}

\begin{prop}\label{Prop:NotQuasiSimpleE8}
An indecomposable preprojective representation of $Q$  has property (C). 
\end{prop}
\begin{proof}
Let $Y$ be an indecomposable preprojective representation of $Q$. If $Y$ is projective, then it is thin and hence it has property (C). If $Y$ is quasi--simple, then $Y$ has property (C) by Proposition~\ref{Prop:Defect23E8}. Suppose that $Y$ is not projective and not quasi--simple. Then there is an irreducible monomorphism $\iota:X\hookrightarrow Y$ ending in $Y$ and starting in a module $X$ such that $|\partial(X)|<|\partial(Y)|$ (see table~\ref{Fig:ExtendedDynkinDiagrams}). Let 
$S=\textrm{Coker}(\iota)$. Since $\iota$ is irreducible, $S$ is quasi--simple and preprojective. Moreover $[X,Y]=1$ and hence, by Unger Lemma~\ref{Lem:Unger}, $[S,X]^1=1$. The short exact sequence 
$$
\xi:\xymatrix{
0\ar[r]&X\ar^\iota[r]&Y\ar[r]&S\ar[r]&0
}
$$
is thus generating and we can apply  Theorem~\ref{Thm:Reduction2}. 
By proposition~\ref{Prop:Defect23E8}, $S$ has property (C). If $X$ is quasi--simple then $\xi$ is almost split and hence $Y$ has property (C) (see corollary~\ref{Cor:RedThmAlmostSplit}). If $X$ is not quasi--simple, it is not a branch module, for defect reasons, and hence, by Lemma~\ref{Lem:XsSXSectionalPath},  $X_S$ is indecomposable and the embedding $X_S\rightarrow X$ is irreducible. There are affine bundles
$$
\xymatrix@R=2pt{
\mathcal{S}^{\xi} _{\mathbf{f}, \mathbf{g}}\ar@{->>}[r]&\textrm{Gr}_{\mathbf{f}}(X)\times \textrm{Gr}_{\mathbf{g}}(S)\textrm{ for }\mathbf{g}\neq \mathbf{dim}\,S,&
\mathcal{S}^{\xi} _{\mathbf{f}, \mathbf{e}}\ar@{->>}[r]&\mathcal{U}_\mathbf{e}(X,X_S)\textrm{ for }\mathbf{g}=\mathbf{dim}\,S.
}
$$
By induction, both $\textrm{Gr}_{\mathbf{f}}(X)\times \textrm{Gr}_{\mathbf{g}}(S)$ and $\mathcal{U}_\mathbf{e}(X,X_S)$ admit a cellular decomposition. We conclude that each stratum $\mathcal{S}^{\xi} _{\mathbf{f}, \mathbf{g}}$ admits a cellular decomposition, and hence $Y$ has property (C). 
\end{proof}

\begin{figure}
\begin{center}
\includegraphics[scale=0.9]{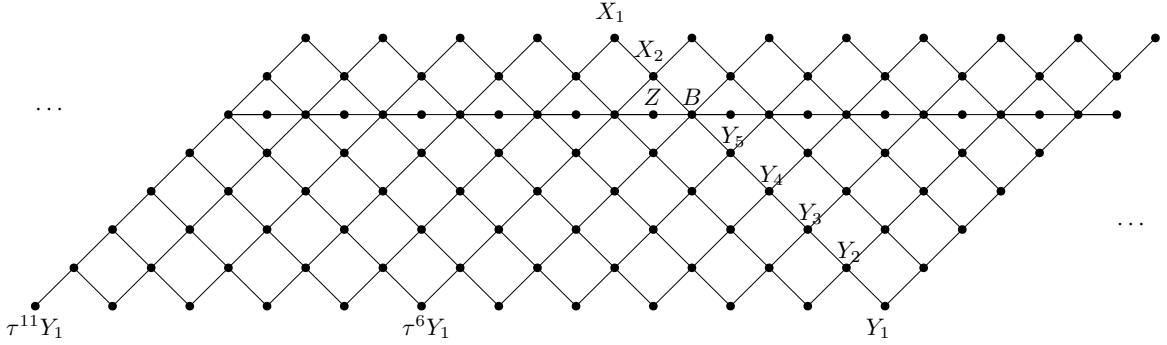}
\end{center}
\caption{Preprojective component of the AR--quiver of a quiver of type $\tilde{E}_8$}\label{Fig:PreprojCompE8}
\end{figure}

\section{The cluster multiplication formula}\label{Sec:ClusterMultFormula}
We recall the definition of the cluster character with coefficients (\cite{DWZ, FuKeller, GLS:07}). 
Let $M$ be a finite dimensional representation of an acyclic quiver $Q$ with $n$ vertices. The $\mathbf{g}$--vector \cite{DWZ} or index \cite{Palu} of $M$ is the integer vector $\mathbf{g}_M\in \ZZ^{Q_0}$ given by $(\mathbf{g}_M)_i:=-\langle S_i,M\rangle$.
Let $H=(h_{ij})_{i,j \in Q_0}$ be the matrix representing the Euler form $\langle-,-\rangle$ in the basis of simples: $\langle M,N\rangle={}^t(\mathbf{dim}\,M) H (\mathbf{dim }N)$.  The \emph{exchange matrix} of $Q$ is defined as
$B:=H-H^t$.  The CC--map is a map $M\mapsto CC(M)$ that associates to a representation of $Q$ a Laurent polynomial $CC(M)\in \ZZ[y_1,\cdots, y_n, x_1^{\pm1},\cdots, x_n^{\pm 1}]$. Given an integer vector $\mathbf{b}=(b_1,\cdots, b_n)$ we use the shorthand notation $\mathbf{x}^\mathbf{b}:=x_1^{b_1}x_2^{b_2}\cdots x_n^{b_n}$.  The Laurent polynomial $CC(M)$ is defined as follows
$$
CC(M):=\sum_{\mathbf{e}}\chi(\Gr_\mathbf{e}(M))\mathbf{y}^\mathbf{e}\mathbf{x}^{B\mathbf{e}+\mathbf{g}_M}.
$$
Given two $Q$--representations $X$ and $S$ such that $[S,X]^1=1$,  a non--zero vector $\xi'\in \Ext^1(S^X, X/X_S)$ is a generalized almost split sequence and there exists an exact sequence
$
0\rightarrow X/X_S\rightarrow \tau S^X\rightarrow I\rightarrow 0
$ where $I$ is either injective or zero (see Lemma~\ref{Lem:XsSX}). We consider the indecomposable decomposition  $I=I_1^{f_1}\oplus I_2^{f_2}\oplus\cdots\oplus I_n^{f_n}$ of $I$. Let $\mathbf{f}=(f_1,\cdots, f_n)$. We can now state and prove a multiplication formula between cluster characters with coefficients.  

\begin{thm}
Let $\xi\in\Ext^1(S,X)$ be a generating extension with middle term $Y$. Then 
\begin{equation}\label{Eq:ClusterFormula}
CC(X)CC(S)=CC(Y)+\mathbf{y}^{\mathbf{dim}\, S^X} CC(X_S\oplus S/S^X)\mathbf{x}^{\mathbf{f}}.
\end{equation}
If, in addition, $\Ext^1(X,S)=0$ and both $X$ and $S$ are exceptional then formula~\eqref{Eq:ClusterFormula} is an exchange relation between the cluster variables $CC(X)$ and $CC(S)$ for the cluster algebra $\mathcal{A}_\bullet(\Sigma)$ with principal coefficients at the initial seed $\Sigma=(Q,\mathbf{x})$.  \end{thm}
\begin{proof}

By using corollary~\ref{Cor:DecForCluster}, we easily get the following formula:  
$$
CC(X)CC(S)=CC(Y)+\mathbf{y}^{\mathbf{dim}\,S^X}\mathbf{x}^{B\mathbf{dim}\, S^X+\mathbf{g}_{(X/X_S)}+\mathbf{g}_{(S^X)}}CC(X_S\oplus S/S^X).
$$
Since $\mathbf{x}^{-\mathbf{g}_{I_k}}=x_k$ by definition of $\mathbf{g}$--vector, to prove formula \eqref{Eq:ClusterFormula} it remains to check that
\begin{equation}\label{Eq:MultFormProof}
B\mathbf{dim}\, S^X+\mathbf{g}_{X/X_S}+\mathbf{g}_{S^X}=-\mathbf{g}_{I}.
\end{equation}
To prove this we recall some few properties of $\mathbf{g}$--vectors. Given a $Q$--representation $M$, by defintion we have $\mathbf{g}_M=-H\mathbf{dim}\,M$. We define the coindex of $M$ as $\mathbf{g}^M:=-H^t\mathbf{dim}\,M$, where $H^t$ denotes the traspose matrix. If $0\rightarrow A\rightarrow B\rightarrow C\rightarrow 0$ is a short exact sequence, then $\mathbf{g}_{A\oplus C}=\mathbf{g}_{A}+\mathbf{g}_C=\mathbf{g}_B$. If $M$ is indecomposable non--projective, the following formula holds (see e.g. \cite[Sec.~5.1]{Cer17})
$\mathbf{g}^M=-\mathbf{g}_{\tau M}$.
Now \eqref{Eq:MultFormProof} follows at once from these properties and concludes the proof of  \eqref{Eq:ClusterFormula}. 

If $X$ and $S$ are exceptional, and $\Ext^1(X,S)=0$, then the two modules $Y$ and $X_S\oplus S/S^X$ are rigid (see Proposition~\ref{Prop:XsSXRigidity}) and hence $CC(Y)$ and $CC(X_S\oplus S/S^X)$ are cluster monomials.  Moreover $X\oplus Y\oplus X_S\oplus S/S^X$ and  $S\oplus Y\oplus X_S\oplus S/S^X$ are rigid, too. To conclude we need to check 
that $[X, I]=[S, I]=[Y, I]=[X_S,I]=[S/S^X, I]=0$.  This is done easily and concludes the proof.  
\end{proof}

\begin{rem}
In case $[X,S]^1=0$ and both $X$ and $S$ are exceptional, formula~\eqref{Eq:ClusterFormula} makes clear that that the $\mathbf{c}$-vector of this exchange relation is the dimension vector of $S^X$.  In particular, it is a positive real Schur root, since $S^X$ is a rigid brick (see Proposition~\ref{Prop:XsSXRigidity}). This provides a representation theoretic interpretation of those $\mathbf{c}$-vectors. The $\mathbf{c}$--vectors of cluster algebras have been studied by several authors: \cite{StellaNakanishi} (for finite type cluster algebras), \cite{SpeyerThomas:AcyclicClusterAlgebras} and \cite{Chavez:CVectors} (for general skew--symmetric cluster algebras). More recently, $\mathbf{c}$--vectors have appeared in $\tau$--tilting theory \cite{DIJ, Asai}.
\end{rem}
\begin{ex}
Let $M$ be a rigid indecomposable non--projective representation of an acyclic quiver. Let $\xi:0\rightarrow \tau M\rightarrow E\rightarrow M\rightarrow 0$ be the almost split sequence ending in $M$. Then 
\begin{equation}\label{Eq:CCAlmostSplit}
CC(M)CC(\tau M)=CC(E)+\mathbf{y}^{\mathbf{dim}\,M}
\end{equation}
\end{ex}
Indeed, since $M$ is exceptional, $[M,\tau M]^1=[\tau M,\tau M]=[M,M]=1$. The generating extension $\xi\in\Ext^1(M,\tau M)$ is an almost split sequence. It follows that $M^{\tau M}=M$ and $(\tau M)_M=M$ (see definition~\ref{Def:XsSx}). Thus formula~\ref{Eq:CCAlmostSplit} is a particular case of the cluster multiplication formula~\ref{Eq:CCAlmostSplit}. For almost split sequences, formula~\eqref{Eq:CCAlmostSplit} holds in full generality \cite{Geiss:14}. 

\section*{Acknowledgments}
We thank Antonio Rapagnetta, Corrado De Concini, Claus Michael Ringel, Bernhard Keller and Julia Sauter for helpful discussions and comments.

	\bibliographystyle{abbrv}
	\bibliography{Literature}
\end{document}